\documentclass[12pt,reqno]{amsart}
\usepackage{amsmath,amssymb,latexsym,esint,cite,mathrsfs,slashed, dsfont}
\usepackage{verbatim,cite,relsize,color,soul}
\usepackage[latin1]{inputenc}
\usepackage{microtype}
\usepackage{color,enumitem,graphicx}
\usepackage[colorlinks=true,urlcolor=blue, citecolor=red,linkcolor=blue,
linktocpage,pdfpagelabels, bookmarksnumbered,bookmarksopen]{hyperref}
\usepackage[hyperpageref]{backref}
\usepackage[english]{babel}
\usepackage{soul}

\usepackage{geometry}
\numberwithin{equation}{section}
\newtheorem{theorem}{Theorem}[section]
\newtheorem{lemma}[theorem]{Lemma}

\newtheorem{proposition}[theorem]{Proposition}

\theoremstyle{definition}

\newtheorem{remark}[theorem]{Remark}
\newtheorem*{ackno}{Acknowledgements}

\newcommand{\Z}{\mathbb Z}

\newcommand{\R}{\mathbb R}

\newcommand{\Ac}{\mathcal A}

\newcommand{\Mca}{\mathcal M}

\newcommand{\vareps}{\varepsilon}
\newcommand{\nnabla}{\slashed{\nabla}}

\def\({\left(}
\def\){\right)}
\def\<{\left\langle}
\def\>{\right\rangle}

\DeclareMathOperator*{\ima}{Im}
\DeclareMathOperator*{\rea}{Re}
\newcommand{\scal}[1]{\left\langle #1 \right\rangle}
\DeclareMathOperator*{\esssup}{ess\,sup}

\title[Scattering for NLS with combined nonlinearities]
{Scattering for non-radial  3D NLS with combined nonlinearities:  the interaction Morawetz approach}
\author[J. Bellazzini, V. D. Dinh, and L. Forcella
]{Jacopo Bellazzini, Van Duong Dinh, and Luigi Forcella}

\address{Jacopo Bellazzini, Dipartimento di Matematica, Universit\`a Degli Studi di Pisa, Largo Bruno Pontecorvo, 5, 56127, Pisa, Italy}
\email{jacopo.bellazzini@unipi.it}

\address{Van Duong Dinh, Ecole Normale Sup\'erieure de Lyon \& CNRS, UMPA (UMR 5669), France
and 
Department of Mathematics, HCMC University of Education, 280 An Duong Vuong, Ho Chi Minh City, Vietnam}
\email{contact@duongdinh.com}

\address{Luigi Forcella, Dipartimento di Matematica, Universit\`a Degli Studi di Pisa, Largo Bruno Pontecorvo, 5, 56127, Pisa, Italy}
\email{luigi.forcella@unipi.it}



\begin{document}
	
	\begin{abstract}
	We give  a new proof of the scattering below the ground state energy level for a class of nonlinear    Schr\"odinger equations (NLS) with  mass-energy intercritical  competing nonlinearities. Specifically, the NLS has a focusing leading order nonlinearity with a defocusing perturbation. Our strategy combines interaction Morawetz estimates \`a la Dodson-Murphy and a new crucial bound for the Pohozaev functional of localized functions, which is essential to overcome the lack of a monotonicity condition. Furthermore, we  give the rate of blow-up for symmetric solutions.  
	\end{abstract}

	\maketitle

	\section{Introduction}
	\label{S1}
	\setcounter{equation}{0}
The aim of the  present paper is to study long time dynamics of solutions to the following nonlinear Schr\"odinger equation with competing nonlinearities
\begin{equation}\label{cNLS}
i\partial_{t}u+\Delta u=\lambda_1|u|^{q-1}u+\lambda_2|u|^{p-1}u,
\end{equation}
where  $u(t,x): I \times \mathbb{R}^d \rightarrow \mathbb{C}$, $I\subseteq \mathbb{R}$, and the parameters $\lambda_1,\lambda_2\in\R$. 
Equation \eqref{cNLS} is a nonlinear Schr\"odinger equation which arises in many physical contexts, and we refer the reader to \cite{Fib, KA2003, KAPS, GFTC} for motivations and further discussions on the models. Since the early works by Zhang \cite{Zhang}, Tao, Vi\c san, and Zhang \cite{TVZ}, and Miao, Xu, and Zhao \cite{MXZ}, equations of the type \eqref{cNLS} have attracted a lot of attention leading to a wide literature concerning different problems: local and global theory, scattering, blow-up, stability of standing waves, and so on; see \cite{AIKN-DIE, AIKN-SM, CMZ, Cheng, CZZ, KOPV, FH, JJTV, SJFA, SJDE, CS, LRN, Luo-JFA, Luo-AIHPC}  and references therein. 

\medskip 

In this article, we are primarily interested in the 3D physical case,  and we consider exponents satisfying $\frac73 <q<p<5$, with a positive coefficient $\lambda_1$ and a negative coefficient  $\lambda_2$. Namely, the nonlinearities are defocusing and mass supercritical, and focusing and energy subcritical, respectively. Without loss of generality, we can normalise the coefficients, by scaling and homogeneity,  as $\lambda_1=-\lambda_2=1.$ Therefore, our main focus lies in the Cauchy problem below: 
\begin{equation} \label{NLS}
		\begin{cases}
			i\partial_t u +\Delta u =  |u|^{q-1} u-|u|^{p-1} u\\ 
			u(0,x)= u_0 \in H^1(\R^3),
		\end{cases}
\end{equation}
where $u(t,x): I \times \mathbb{R}^3 \rightarrow \mathbb{C}$, $I\subseteq \mathbb{R}$, with  $\frac{7}{3}<q<p<5$.

\medskip 

At the local level, it is well-known that \eqref{NLS} admits solutions, see \cite{Cazenave}, and by denoting $I_{\max}\ni0$ the maximal interval of existence,  solutions preserve (among other quantities) mass and energy, defined by
\begin{equation}\label{def:mass}
M(u(t)):=\int_{\mathbb R^3}|u(t)|^2\,dx,
\end{equation}
and
\begin{equation}\label{def:en}
E(u(t)):=\frac{1}{2}\int_{\mathbb R^3}|\nabla u(t)|^2\,dx+\frac{1}{q+1}\int_{\mathbb R^3}|u(t)|^{q+1}\,dx-\frac{1}{p+1}\int_{\mathbb R^3}|u(t)|^{p+1}\,dx,
\end{equation}
respectively. Hence \eqref{def:mass} and \eqref{def:en} are independent of time. 

\medskip

The main purpose of this paper is to study the energy scattering for \eqref{NLS} with non-radial initial data which avoids the machinery of the concentration-compactness and rigidity method. Before stating our main contribution, let us briefly recall known results related to the energy scattering for combined NLS with focusing leading nonlinearity. 

\subsection{Known results}
Since the pioneering work by Kenig and Merle \cite{KM} establishing the scattering below the ground state threshold for the energy critical focusing (3D quintic) NLS with radial data, i.e. \eqref{cNLS} with $\lambda_1=0$ and $\lambda_2<0$, the literature on the scattering problem for nonlinear dispersive PDEs rapidly grew up. The Kenig-Merle approach relies on the so-called concentration-compactness and rigidity scheme, based on profile decompositions. The radial assumption on the milestone paper \cite{KM} was removed by Killip and Vi\c san \cite{KV} for dimensions greater than or equal to five and by Dodson \cite{Dodson-ASENS} in dimension four. After the energy critical cases, the mass-critical problems have been considered by several mathematicians. Killip, Tao, and Vi\c san \cite{KTV} proved the scattering below the mass threshold for the 2D cubic NLS with radial data. The higher dimensional cases, still with radial data, was investigated by Killip, Vi\c san, and Zhang \cite{KVZ}. Dodson \cite{Dodson-AM} completely removed the radial requirement and extended these results to all dimensions.

Moving to the mass-supercritical and energy-subcritical nonlinearities (the so-called intercritical case), still by exploiting a Kenig-Merle scheme, the scattering below the ground state threshold was established by Holmer and Roudenko \cite{HR} for the 3D cubic NLS with radial data. Then the radial condition was later removed by Duyckaerts, Holmer, and Roudenko \cite{DHR}. Extensions to higher dimensions, in the whole intercritical range, were done by Akahori and Nawa \cite{AN}, and Fang, Xie, and Cazenave \cite{CXF}. 

\medskip 

The argument used in the previously cited papers follows the Kenig-Merle road map with a concentration-compactness and rigidity scheme. An alternative strategy to prove scattering below the ground state energy in a non-radial framework was recently introduced by Dodson and Murphy \cite{DM-MRL} (see \cite{DM-PAMS} by the same authors for the radial case) for the $\dot{H}^{1/2}$-critical nonlinearity, combining interaction Morawetz estimates together with a scattering criterion. See also \cite{Dinh-DCDS} for a generalisation to intercritical powers in general dimension.

\medskip 

As far as we know, similarly to the single nonlinearity, also in the case of double nonlinearity all the existing results rely on a concentration-compactness and rigidity scheme.  For equations of the type \eqref{cNLS} with a focusing leading term, we mention here 
the papers by Akahori, Ibrahim, Kikuchi, and Nawa \cite{AIKN-SM} with energy critical leading term and a focusing mass-supercritical perturbation in dimensions $d\geq 5$ and the recent extension to $d\geq 3$, given by Luo in \cite{Luo-JFA}. In the case of a focusing mass-critical perturbation, the scattering was recently proved by Luo \cite{Luo-AIHPC} in all dimensions $d\geq 3$. Note that the 3D case requires the radial assumption due to the fact that the non-radial scattering for the 3D quintic NLS is still an open question. For the defocusing mass-supercritical perturbation, an early scattering result was  proved by Miao, Xu, and Zhao \cite{MXZ} for the 3D cubic-quintic NLS with radial data. It was extended to dimensions five and higher in \cite{MXZ-15}. For the defocusing mass-critical perturbation, Cheng, Miao, and Zhao \cite{CMZ} proved the scattering with radial data in dimensions $d\leq 4$. It was recently extended, still with radial data, to dimensions $d\geq 5$ by Luo \cite{Luo-JFA}. When the leading term is mass-supercritical and energy-subcritical  we refer to Cheng \cite{Cheng}  with the defocusing mass-critical perturbation.

\medskip

Concerning the intercritical range on nonlinearities, the energy scattering was proved by Akahora, Kikuchi, and Nawa in \cite{AKN}, and as for the already mentioned papers, the scattering for such equations is proved by means of a Kenig-Merle approach. In fact, in \cite{AKN} a wide class of NLS-type equations is treated. Among them, one also finds the NLS equation with combined nonlinearities with intercritical exponents.  The main novelty in \cite{AKN} is that they replace a usual monotonicity condition (which is in fact not satisfied by combined power nonlinearities) with weaker assumptions. Then, in order to prove a desired lower bound on the Pohozaev functional $G$, see \eqref{Poho-funct} below, Akahori, Kikuchi, and Nawa distinguish  different cases depending on sign of $G$ along minimizing sequences. Specifically, the tool given by  \cite[Lemma 3.3]{AKN} enables them to employ a concentration-compacteness and rigidity scheme. 

\medskip

In our paper, we propose an alternative approach to prove the energy scattering of \cite{AKN} for \eqref{NLS} based on the recent method proposed by Dodson and Murphy, see \cite{DM-MRL}, for the 3D focusing cubic NLS, which relies on suitable interaction Morawetz estimates, and avoids the Kenig-Merle machinery. In the next sections, we explain our new contributions and the main novelties.

\subsection{Main results and novelties}
To state our main results and our original achievements, we begin by  introducing the following Pohozaev functional
\begin{equation}\label{Poho-funct}
	G(u)=\int_{\mathbb R^3}|\nabla u(t)|^2\,dx+\frac{3}{2}\left(\frac{q-1}{q+1}\right)\int_{\mathbb R^3}|u(t)|^{q+1}\,dx-\frac{3}{2}\left(\frac{p-1}{p+1}\right)\int_{\mathbb R^3}|u(t)|^{p+1}\,dx,
\end{equation}
which will be crucial for the characterization of the dynamics of solutions to \eqref{NLS}. We note in particular that solutions to  \eqref{NLS} of the form $u(t,x)=e^{i \omega t}\psi(x)$, with $\omega$ being a real parameter, i.e., standing wave solutions,  fulfill $G(\psi)=0$.

Furthermore, we  introduce the action functional
\[
\begin{aligned}
	S_\omega(\phi) &= E(\phi)+\frac{\omega}{2} M(\phi) \\
	&= \frac{1}{2}\|\nabla \phi\|^2_{L^2(\R^3)} -\frac{1}{p+1}\|\phi\|^{p+1}_{L^{p+1}(\R^3)} + \frac{1}{q+1}\|\phi\|^{q+1}_{L^{q+1}(\R^3)} + \frac{\omega}{2}\|\phi\|^2_{L^2(\R^3)}
\end{aligned}
\]
and  the following ground state energy
\[
m_\omega := \inf \left\{S_\omega (\phi) : \phi \in H^1(\R^3)\backslash \{0\}, \ G(\phi)=0\right\}.
\]
Note that the action functional $S_\omega$ is a conserved quantity along the time evolution of a solution to \eqref{NLS}, as sum of conserved quantities. 

\medskip

The existence of standing waves of the form $u(t,x)=e^{i \omega t}\psi(x)$ can be proved by showing that the aforementioned   ground state energy is achieved, i.e., that the infimum  of the action on the constraint $G=0$  is indeed a minimum.
The fact that ground state energy level $m_\omega$ is achieved for any $\omega>0$ is proved in \cite{AKN}  while the instability properties of the corresponding standing waves is a recent result by \cite{FH}. It is worth mentioning that $m_\omega$ is given by $S_\omega(Q_\omega)$, where $Q_\omega$ solves the elliptic equation 
\[
-\Delta Q_\omega+\omega Q_\omega+Q_\omega^q-Q_\omega^p=0.
\] 
The question concerning the existence of ground states with an assigned mass has been addressed in \cite{BFG}. 
Eventually, for  $\omega>0$, we introduce
\[
\Ac^+_\omega:= \left\{ u\in H^1(\R^3) : S_\omega(u)<m_\omega, \ G(u) \geq 0\right\}
\]
and
\[
\Ac^-_\omega:= \left\{u \in H^1(\R^3): S_\omega(u) <m_\omega, \ G(u) <0\right\},
\]
and we recall the notion of scattering: we say that a solution 
$u(t,x)$ to \eqref{NLS}  \emph{scatters} provided that
\begin{equation*}
	\lim_{t\rightarrow \pm \infty} \|u(t) - e^{it\Delta} u_0^\pm \|_{H^1 (\mathbb{R}^3)}=0,
\end{equation*}
for suitable $u_0^\pm\in H^1(\mathbb{R}^3)$, where $e^{it\Delta}$ is the  linear Schr\"odinger propagator. 

\medskip 

\noindent Our purpose is to show that:
\begin{itemize}
	\item[(1)] if $u_0\in \Ac^+_\omega$, then the corresponding solution exists globally in time and scatters in $H^1(\R^3)$ in both directions;
	\item[(2)] if $u_0 \in \Ac^-_\omega$, then  if we assume $u_0$ is radially symmetric or cylindrically symmetric (with some additional restrictions on $p$ and some relaxations on the hypothesis on $q$), then the corresponding solution blows up in finite time with an explicit blow-up rate. 
\end{itemize}

\medskip 

\noindent The first result concerns the energy scattering for \eqref{NLS} with data in $\Ac^+_\omega$. 

\begin{theorem} \label{theo-scat}
	Let $\frac{7}{3}<q<p<5$ and $\omega>0$. Let $u_0 \in \Ac^+_\omega$. Then the corresponding solution to \eqref{NLS} exists globally in time and scatters.
\end{theorem}

\subsubsection{Comments on Theorem \ref{theo-scat}}  In order to prove Theorem \ref{theo-scat}, we combine a scattering criterion jointly with interaction Morawetz estimates  in the spirit of \cite{DM-MRL}, and our new key ingredient is a coercivity property established in Subsection \ref{subsec:coer} for solutions belonging to $\Ac^+_\omega$. Specifically, we provide here the new crucial bound \eqref{intro:loc-poho} on the Pohoazev functional for localized solutions, which enables us to build upon \cite{DM-MRL}. In order to get this property, we perform a careful variational study of the functional $G$ that we believe it is of independent interest.  More precisely, we use the variational properties of the functional $G$, that  appears in the virial-like estimates, to prove an interaction Morawetz estimate. 
We show that taken a suitable cut-off function $\chi_R(x)=\chi(\frac{x}{R})$ such that $\chi_R= 1$ when $|x|\leq (1-\eta)R$ with some $\eta>0$ small, the Pohozaev functional fulfills, for any $R$ sufficiently large,
\begin{equation}\label{intro:loc-poho}
	G(\chi_R(\cdot-z) u^\xi(t, \cdot)) \geq \delta \|\nabla (\chi_R(\cdot-z) u^\xi(t, \cdot))\|^2_{L^2(\mathbb R^3)}
\end{equation}
for any space shift $z\in \R^3$, all time $t\in \R$, and some $\xi =\xi(t,z,R) \in \R^3$, where $u^\xi(t,x)$ stands for the modulated function $e^{ix\cdot \xi}u(t,x)$. Here $\delta>0$ is a constant independent of time and the translation vector.

\begin{remark}\rm The localized coercivity  result, i.e., that the coercivity is true if we localize the modulated and translated solution $u(t)$ in a sufficiently  large ball centered anywhere,  as far as we know is a new  property which may  be relevant in other contexts. Indeed, \eqref{intro:loc-poho} may be used in contexts where a concentration-compactness scheme is unknown to be exploitable. We refer the reader to the very recent work by Luo, \cite{Luo-arXiv}, where our estimate \eqref{intro:loc-poho} is crucially used for the scattering of solutions to intercritical NLS posed on high-dimensional waveguides. Hence, we believe that our strategy is not only a mere alternative proof of the results of \cite{AKN} for \eqref{NLS}.

\end{remark}
\begin{remark}\rm
We shall emphasize  that in the paper by Dodson-Murphy, the  coercivity property needed to prove interaction Morawetz estimates comes from the refined Gagliardo-Nirenberg inequality (see \cite[Lemma 2.1]{DM-MRL})
\begin{equation}\label{refined-GN}
	\|f\|_{L^4(\R^3)}^4\leq C_{GN}\|f\|_{L^2(\R^3)}\|\nabla f\|_{L^2(\R^3)}\|\nabla f^\xi\|_{L^2(\R^3)}^2
\end{equation}
which holds for any $f\in H^1(\R^3)$ and any $\xi\in\R^3$, and that we are prevented to use in our setting. Indeed, due to the non-homogeneity of nonlinearities,  inequality \eqref{refined-GN} is not applicable in our framework. Therefore, our proof is not a straightforward adaptation of \cite{DM-MRL} as we need to bypass the difficulties related to the lack of scaling due to the presence of two competing nonlinearities.
\end{remark}

\begin{remark}\rm
One of the motivation to study the energy scattering for \eqref{NLS} came to complete the picture on the dynamics of solutions to \eqref{NLS} after a recent work by the first and third authors with Georgiev, see \cite{BFG}, where existence of normalized ground states for \eqref{NLS} and blow-up results were established in the 3D space. We keep the presentation for the physically relevant 3D case as a counterpart of the results we established in \cite{BFG}.
It is worth mentioning that the scattering results of Theorem \ref{theo-scat} still holds in arbitrary space dimension, as \eqref{intro:loc-poho} can be established in the same way in any dimension $d\geq1$, for any intercritical powers $\frac4d+1<q<p<1+\frac{4}{(d-2)^+}$, with $(d-2)^+:=\max\{0,d-2\}$. 
\end{remark}

\medskip

Our second  result concerns the blow-up rate for finite time blowing-up solutions to \eqref{NLS} with initial data in $\Ac^-_\omega$. 
In \cite{BFG}, we proved that if  $\frac{7}{3}<p<5$,  $1<q<p$, and $\omega>0$,  for $u_0 \in \Ac^-_\omega$, if one of the following conditions is satisfied:
\begin{itemize}
	\item $u_0 \in \Sigma(\R^3):=H^1(\R^3)\cap L^2(\mathbb R^3,|x|^2 dx)$,
	\item $u_0 \in H^1(\R^3)$ is radial,
	\item $u_0 \in \Sigma_3(\R^3)$ and $p \leq 3$, where
	\[
	\Sigma_3(\R^3):= \left\{f\in H^1(\R^3) : f(x)=f(\overline{x}, x_3) = f(|\overline{x}|, x_3), \ f \in L^2(\R^3, x_3^2dx) \right\}, \quad \overline{x}=(x_1,x_2),
	\]
\end{itemize}
then the corresponding solution to \eqref{NLS} blows up in finite time.

\medskip

Here we show that we have the following blow-up rate for \eqref{NLS}.

\begin{proposition} \label{prop-blow-rate}
	Let $\frac73<p<5$, $1<q<p$, $\omega>0$, and $u_0 \in \Ac^-_\omega$.   \medskip
	
	\noindent \text{(i)}	Assume that  one of the following conditions is fulfilled:
	\begin{itemize}
		\item[\textit{(1)}] $u_0 \in H^1(\R^3)$ is radial,
		\item[\textit{(2)}] $u_0 \in \Sigma_3(\R^3)$ and $p<3$.
	\end{itemize}
	Then the corresponding solution to \eqref{NLS} blows up in finite time, i.e., $T^*<+\infty$, and  for $t$ close to $T^*$,
	\begin{align} \label{est-blow}
		\int_t^{T^*} (T^*-\tau) \|\nabla u(\tau)\|^2_{L^2(\R^3)}d\tau  \leq \left\{
		\begin{array}{ccl}
			C (T^*-t)^{\frac{2(5-p)}{p+3}} &\text{if}& \textit{(1)} \text{ holds}, \\
			C (T^*-t)^{\frac{4(3-p)}{5-p}} &\text{if}& \textit{(2)} \text{ holds}.
		\end{array}
		\right.
	\end{align}
	In addition, there exists a time sequence $t_n \nearrow T^*$ such that
	\begin{align} \label{blow-rate}
		\|\nabla u(t_n)\|_{L^2(\R^3)} \leq \left\{
		\begin{array}{ccl}
			C(T^*-t_n)^{-\frac{2(p-1)}{p+3}} &\text{if} &  \textit{(1)} \text{ holds}, \\
			C(T^*-t_n)^{-\frac{p-1}{5-p}} &\text{if} &  \textit{(2)}\text{ holds}. 
		\end{array}
		\right.
	\end{align}
	\medskip
	\text{(ii)} If  we do not assume any symmetry on the solution (nor any restriction on the exponents), then either $T^*<\infty$ or $u$ grows up in infinite time, namely $T^*=\infty$ and $\limsup_{t\to\infty}\|\nabla u(t)\|_{L^2(\R^3)}=\infty$. 
\end{proposition}

The Proposition above is based on the results we obtained in \cite{BFG}, jointly with a Merle, Rapha\"el, and Szeftel argument  \cite{MRS}. In particular, the new results  are the blow-up rate estimates \eqref{est-blow} and \eqref{blow-rate} of point $(i)$, and point $(ii)$. The finite time blow-up for symmetric solutions is established in \cite{BFG}, and we include it in the statement for sake of completeness. The grow-up result follows by means of a localized properties of the mass, jointly with a contradiction argument as in \cite{DWZ}.

\subsection*{Notations} In the rest of the paper, we will use the notations below. \medskip

\noindent Given two quantities $A$ and $B,$ we denote $A \lesssim B$ if there exists a positive constant $C$ such that $A \leq CB$. 
For $1\leq p\leq \infty,$ the $L^p=L^p(\Omega; \mathbb C)$ are the classical Lebesgue spaces endowed with norm $\|f\|_{L^p(\Omega)}=\left(\int_{\Omega}|f(x)|^p\,dx\right)^{1/p}$ if $p\neq\infty$ or $\|f\|_{L^\infty(\Omega)}=\esssup_{x\in \Omega}|f(x)|$ for $p=\infty.$ We denote by $H^1=H^1(\mathbb R^3;\mathbb C)$ the usual Sobolev space of $L^2$ functions with gradient in $L^2$. Here $\langle \cdot \rangle$ stands for the Japanese brackets $\langle  \cdot \rangle=(1+|\cdot|^2)^{1/2}$.  
Given an interval $I\subseteq \mathbb R,$ bounded or unbounded, we define by $L^p_tX_x=L^p_t(I,X_x)$ the Bochner space of vector-valued functions $f:I\to X$ endowed with the norm $\left(\int_{I}\|f(s)\|_{X}^p\,ds\right)^{1/p}$ for $1\leq p<\infty,$ with similar modification as above for $p=\infty.$ (In what follows, $f\in L^p_tX_x$ means that $f=f(t,x)$ is a function depending on the time variable $t\in I\subseteq \R$ and the space variable $x\in\R^3,$ with finite $L^p_tX_x$-norm). For any $p\in[1,\infty]$, $p^\prime$ denotes its dual defined by $p^\prime=\frac{p}{p-1}$. As we work in the 3D case, we omit the $\R^3$ notation when no confusion may arise.
 	
	\section{Variational analysis}
	\label{S2}
	\setcounter{equation}{0}
In this section, we recall and prove some crucial variational tools used along the paper, used in particular to define  the scattering/blow-up dichotomy regions of initial data for the Cauchy problem \eqref{NLS}. Furthermore, some of the results illustrated below will be also essential in proving a new coercivity property that we need to prove scattering by means of suitable interaction Morawetz  estimates.  \\

	Let $\omega>0$. We consider the minimization problem
	\begin{align*}
	m_\omega := \inf \left\{S_\omega (\phi) : \phi \in H^1\backslash \{0\}, \ G(\phi)=0\right\},
	\end{align*}
	where 
	\[
	S_\omega(\phi) = E(\phi)+\frac{\omega}{2} M(\phi) = \frac{1}{2}\|\nabla \phi\|^2_{L^2} -\frac{1}{p+1}\|\phi\|^{p+1}_{L^{p+1}} + \frac{1}{q+1}\|\phi\|^{q+1}_{L^{q+1}} + \frac{\omega}{2}\|\phi\|^2_{L^2}
	\]
	is the action functional and
	\[
	G(\phi) = \|\nabla \phi\|^2_{L^2} -\frac{3(p-1)}{2(p+1)}\|\phi\|^{p+1}_{L^{p+1}} +\frac{3(q-1)}{2(q+1)}\|\phi\|^{q+1}_{L^{q+1}}
	\]
	is the Pohozaev functional.
	
	\begin{proposition} \label{prop-mini-m-omega}
		Let $\frac{7}{3}<p<5$, $1<q<p$, and $\omega>0$. Then $m_\omega>0$ and there exists at least a minimizer for $m_\omega$. 
	\end{proposition}
	
	Before giving the proof of Proposition \ref{prop-mini-m-omega}, let us start with the following observation. 
	\begin{lemma} \label{lem-lamb-0}
		Let $\phi \in H^1\backslash\{0\}$. Then there exists a unique $\lambda_0>0$ such that
		\[
		G(\phi_\lambda)~ \left\{ 
		\begin{array}{cl}
		>0 &\text{if } 0<\lambda<\lambda_0, \\
		=0 &\text{if } \lambda=\lambda_0, \\
		<0 &\text{if } \lambda>\lambda_0,
		\end{array}
		\right.
		\]
		where
		\begin{align}\label{scaling}
		\phi_\lambda(x):= \lambda^{\frac{3}{2}} \phi(\lambda x), \quad \lambda>0.
		\end{align}
	\end{lemma}
	
	\begin{proof}
		We have
		\begin{align*}
		G(\phi_\lambda) &= \lambda^2\|\nabla \phi\|^2_{L^2} -\frac{3(p-1)}{2(p+1)} \lambda^{\frac{3}{2}(p-1)} \|\phi\|^{p+1}_{L^{p+1}} + \frac{3(q-1)}{2(q+1)}\lambda^{\frac{3}{2}(q-1)} \|\phi\|^{q+1}_{L^{q+1}} \\
		&=\lambda^2 \(\|\nabla \phi\|^2_{L^2} -\frac{3(p-1)}{2(p+1)} \lambda^{\frac{3}{2}(p-1)-2} \|\phi\|^{p+1}_{L^{p+1}} + \frac{3(q-1)}{2(q+1)}\lambda^{\frac{3}{2}(q-1)-2} \|\phi\|^{q+1}_{L^{q+1}}\) \\
		&=:\lambda^2 f(\lambda).
		\end{align*}
		Consider $f(\lambda)=a - b\lambda^\alpha + c\lambda^{\beta}$ with $a,b,c>0$, $\alpha>0$, and $\beta<\alpha$. We have
		\[
		f'(\lambda) = \lambda^{\beta-1} (c\beta - b\alpha \lambda^{\alpha-\beta}).
		\]
		In particular, there exists a unique $\lambda_1>0$ such that 
		\[
		f(\lambda_1)= \max_{\lambda>0} f(\lambda). 
		\]
		Drawing the graph of $f$, we see that there exists a unique $\lambda_0>0$ such that $f(\lambda_0)=0$. Moreover, $f(\lambda)>0$ for $0<\lambda<\lambda_0$ and $f(\lambda)<0$ for $\lambda>\lambda_0$. This shows the lemma.
	\end{proof}
	
	\begin{proof}[Proof of Proposition \ref{prop-mini-m-omega}] We proceed in several steps, illuminated by the works of Ibrahim, Masmoudi, and Nakanishi \cite{IMN}, and Akahori, Ibrahim, Kikuchi, and Nawa \cite{AIKN-DIE}. \\
		
\noindent		{\bf Step 1. An auxiliary minimization problem.} Denote
		\[
		I_\omega(\phi):= S_\omega(\phi)- \frac{2}{3(q-1)} G(\phi) = \frac{3q-7}{6(q-1)} \|\nabla \phi\|^2_{L^2} + \frac{p-q}{(p+1)(q-1)}\|\phi\|^{p+1}_{L^{p+1}} +\frac{\omega}{2} \|\phi\|^2_{L^2}
		\]
		and consider
		\begin{align} \label{m-omega-tilde}
		\tilde{m}_\omega:= \inf \left\{I_\omega(\phi) : \phi \in H^1\backslash \{0\},\ G(\phi)\leq 0\right\}.
		\end{align}
		We claim that $m_\omega=\tilde{m}_\omega>0$. In fact, it is clear that $\tilde{m}_\omega \leq m_\omega$. Now let $\phi \in H^1\backslash \{0\}$ be such that $G(\phi)\leq 0$. By Lemma \ref{lem-lamb-0}, there exists $\lambda_0 \in (0,1]$ such that $G(\phi_{\lambda_0})=0$. Thus
		\[
		m_\omega \leq S_\omega(\phi_{\lambda_0}) = I_\omega(\phi_{\lambda_0}) \leq I_\omega(\phi),
		\]
		where we have used $\lambda_0\leq 1$ in the last inequality. Taking the infimum, we get $m_\omega \leq \tilde{m}_\omega$, hence $m_\omega=\tilde{m}_\omega$. 
		
		To see that $\tilde{m}_\omega>0$, we take a minimizing sequence $\{\phi_n\}_n$ for $\tilde{m}_\omega$, i.e., $\phi_n \in H^1\backslash \{0\}$, $G(\phi_n) \leq 0$, and $I_\omega(\phi_n) \to \tilde{m}_\omega$. As $I_\omega$ is non-negative, there exists $C=C(\omega)>0$ such that
		\[
		\|\phi_n\|^2_{L^2} \leq C, \quad \forall n\geq 1.
		\] 
		On the other hand, since $G(\phi_n) \leq 0$, we have
		\[
		\|\nabla \phi_n\|^2_{L^2} \leq \frac{3(p-1)}{2(p+1)}\|\phi_n\|^{p+1}_{L^{p+1}}
		\]
		which together with the standard Gagliardo-Nirenberg inequality yield
		\[
		\|\nabla \phi_n\|^2_{L^2} \leq C\|\nabla \phi_n\|^{\frac{3(p-1)}{2}}_{L^2}\|\phi_n\|^{\frac{5-p}{2}}_{L^2} \leq C\|\nabla \phi_n\|^{\frac{3(p-1)}{2}}_{L^2}.
		\]
		This shows that 
		\begin{align} \label{low-boun-phi-n}
		\|\nabla \phi_n\|^2_{L^2} \geq C>0, \quad \forall n\geq 1,
		\end{align}
		hence $\tilde{m}_\omega >0$. 	\\
		
\noindent		{\bf Step 2. Minimizers for $\tilde{m}_\omega$.} Let $\{\phi_n\}_n$ be a minimizing sequence for $\tilde{m}_\omega$. Observe that if $\phi_n^*$ is the Schwarz symmetrization of $\phi_n$, then $\{\phi_n^*\}_n$ is still a minimizing sequence for $\tilde{m}_\omega$. Thus without loss of generality, we can assume that $\phi_n$ is radially symmetric. As $I_\omega$ is non-negative, we see that $\{\phi_n\}_n$ is a bounded sequence in $H^1$. Passing to a subsequence if necessary, there exists $\phi \in H^1$ such that $\phi_n \rightharpoonup \phi$ weakly in $H^1$ and $\phi_n \to \phi$ strongly in $L^r$ for all $2<r<6$. As $G(\phi_n)\leq 0$, using \eqref{low-boun-phi-n}, we have
		\[
		\|\phi_n\|^{p+1}_{L^{p+1}} \geq \frac{2(p+1)}{3(p-1)} \|\nabla \phi_n\|^2_{L^2} \geq \frac{2(p+1)}{3(p-1)} C >0, \quad \forall n\geq 1
		\] 
		which shows $\phi \ne 0$. We also have
		\[
		G(\phi) \leq \liminf_{n\to \infty} G(\phi_n) =0,
		\]
		thus
		\begin{align*}
		\tilde{m}_\omega \leq I_\omega(\phi) \leq \liminf_{n\to \infty} I_\omega(\phi_n) = \tilde{m}_\omega.
		\end{align*}
		This shows that $\phi$ is a minimizer for $\tilde{m}_\omega$.\\
		
\noindent		{\bf Step 3. Minimizers for $m_\omega$.} We now show that there exists at least a minimizer for $m_\omega$. Let $\phi$ be a minimizer for $\tilde{m}_\omega$. By Lemma \ref{lem-lamb-0}, there exists $\lambda_0\in (0,1]$ such that $G(\phi_{\lambda_0})=0$. Thus
		\[
		m_\omega \leq S_\omega(\phi_{\lambda_0}) =I_\omega(\phi_{\lambda_0}) \leq I_\omega(\phi) =\tilde{m}_\omega = m_\omega.
		\]
		This shows that $\lambda_0=1$, $G(\phi)=0$, and $S_\omega(\phi)=m_\omega$. In particular, $\phi$ is a minimizer for $m_\omega$. 
	\end{proof}

	\section{Scattering criterion}
	\label{S3}
	\setcounter{equation}{0}
In this section, we state and prove a scattering criterion for \eqref{cNLS} which is inspired by the paper of Dodson and Murphy \cite{DM-MRL}. As explained in \cite{DM-MRL}, the criterion infers that if on any time interval which is large enough, one can find a large interval where the scattering norm is small, then the global solution scatters. 

We first introduce the following exponents:
\begin{align} \label{expo-p-1}
	a_1:=\frac{4(p+1)}{3(p-1)}, \quad m_1 := \frac{2(p-1)(p+1)}{5-p}, \quad n_1 := \frac{2(p-1)(p+1)}{3p^2-5p-2}
\end{align}
and
\begin{align} \label{expo-p-2}
b_1:= p+1, \quad r_1:=\frac{6(p-1)(p+1)}{3p^2+2p-13}.
\end{align}
It can be easily checked that $(a_1,b_1)$ and $ (m_1,r_1)$ are Schr\"odinger admissible and
\[
\frac{1}{m_1}+\frac{1}{n_1}=\frac{2}{a_1}, \quad \frac{1}{b_1}=\frac{1}{r_1}-\frac{\sigma_1}{3}, \quad \sigma_1:= \frac{3p-7}{2(p-1)}.
\]

	\begin{proposition} \label{prop-scat-crite}
		Let $\frac{7}{3}<q<p<5$. Suppose that $u(t)$ is a global solution to \eqref{NLS} satisfying $\|u\|_{L^\infty_t(\R, H^1_x)} <\infty$. Then there exist $\vareps>0$ and $T_0 = T_0(\vareps)>0$ such that if for any $a>0$, there exists $t_0 \in (a,a+T_0)$ such that $[t_0-\vareps^{-\sigma}, t_0] \subset (a,a+T_0)$ and
		\begin{align} \label{scat-cond}
		\|u\|_{L^{m_1}_t([t_0-\vareps^{-\sigma}, t_0], L^{b_1}_x)} \lesssim \vareps
		\end{align}
		for some $\sigma>0$, then the solution scatters forward in time.
	\end{proposition}
	
	\begin{proof}
		By Lemma \ref{lem-small-scat}, it suffices to prove that there exists $T>0$ such that
		\begin{align} \label{scat-cond-appl}
		\|e^{i(t-T)\Delta} u(T)\|_{L^{m_1}_t([T,\infty), L^{b_1}_x)} \lesssim \vareps^\mu
		\end{align}
		for some $\mu>0$. To show \eqref{scat-cond-appl}, we write
		\[
		e^{i(t-T)\Delta} u(T) = e^{it\Delta} u_0 + i \int_0^T e^{i(t-s)\Delta} \( |u(s)|^{p-1} u(s) - |u(s)|^{q-1} u(s)\) ds.
		\]
		By Sobolev embedding and Strichartz estimate (see Proposition \ref{prop-str-est}), we have
		\[
		\|e^{it\Delta} u_0\|_{L^{m_1}_t(\R, L^{b_1}_x)}\lesssim \|u_0\|_{H^1_x}.
		\]
		By the monotone convergence theorem, there exists $T_1>0$ sufficiently large such that for all $T>T_1$,
		\begin{align} \label{est-line-part}
		\|e^{it\Delta} u_0\|_{L^{m_1}_t([T,\infty), L^{b_1}_x)} \lesssim \vareps.
		\end{align}
		Taking $a=T_1$ and $T=t_0$ with $a$ and $t_0$ as in \eqref{scat-cond}, we write
		\begin{align*}
		&\int_0^T e^{i(t-s)\Delta} \(|u(s)|^{p-1} u(s) -|u(s)|^{q-1} u(s)\) ds =\\
		& \int_{I_1\cup I_2}  e^{i(t-s)\Delta} \(|u(s)|^{p-1} u(s) -|u(s)|^{q-1} u(s)\) ds 	
		\end{align*}
		where $I_1=[T-\vareps^{-\sigma},T]$ and $I_2=[0,T-\vareps^{-\sigma}]$. By using the linearity, we denote by  $F_1(t)$ and $F_2(t)$ the integrals over $I_1$ and $I_2$, respectively.\\		
		
\noindent 		To estimate $F_1$, we start with the following observation:
		\begin{align} \label{obse-scat-cond}
		\|u\|_{L^{qn'_1}_t(I_1, L^{qb'_1}_x)} \leq \|u\|^\theta_{L^{m_1}_t(I_1, L^{b_1}_x)} \|u\|^{1-\theta}_{L^{\rho}_t(I_1, L^{\gamma}_x)},
		\end{align}
		where
		\[
		\theta = \frac{3pq-3q-4p}{q(3p-7)}, \quad \rho= \frac{(1-\theta)qm_1 n'_1}{m_1-\theta q n'_1}, \quad \gamma=\frac{(1-\theta)q b_1 b'_1}{b_1-\theta qb'_1}.
		\]
		Since $\frac{7}{3}<q<p<5$, we see that 
		\[
		\theta \in (0,1), \quad \frac{2}{\rho}+\frac{3}{\gamma}=\frac{3}{2}, \quad \gamma \in [2,6].
		\]
		In particular, $(\rho,\gamma)$ is a Schr\"odinger admissible pair. Thanks to Strichartz estimates for non-admissible pairs (see Proposition \ref{prop-str-est}), we have
		\begin{align*}
		\|F_1\|_{L^{m_1}_t([T,\infty), L^{b_1}_x)} &\leq C\||u|^{p-1} u\|_{L^{n'_1}_t(I_1, L^{b'_1}_x)} +C\||u|^{q-1} u\|_{L^{n'_1}_t(I_1,L^{b'_1}_x)} \\
		&\leq C \|u\|^p_{L^{m_1}_t(I_1, L^{b_1}_x)} + C\|u\|^q_{L^{qn'_1}_t(I_1, L^{qb'_1}_x)} \\
		&\leq C \|u\|^p_{L^{m_1}_t(I_1, L^{b_1}_x)} + C\|u\|^{q\theta}_{L^{m_1}_t(I_1, L^{b_1}_x)} \|u\|^{q(1-\theta)}_{L^{\rho}_t(I_1, L^{\gamma}_x)} \\
		&\leq C \|u\|^p_{L^{m_1}_t(I_1, L^{b_1}_x)} +C|I_1|^{\frac{q}{\rho}(1-\theta)} \|u\|^{q\theta}_{L^{m_1}_t(I_1,L^{b_1}_x)} \|u\|^{q(1-\theta)}_{L^\infty_t(I_1, H^1_x)}.
		\end{align*}
		Since $\|u\|_{L^\infty_t(\R, H^1_x)} <\infty$, we infer from \eqref{scat-cond} that
		\[
		\|F_1\|_{L^{m_1}_t([T,\infty), L^{b_1}_x)} \leq C\vareps^p + C\vareps^{q\theta - \frac{q}{\rho}\sigma(1-\theta)}.
		\]
		Taking $\sigma>0$ small, we have 
\begin{align} \label{est-F1}
\|F_1\|_{L^{m_1}_t([T,\infty), L^{b_1}_x)} \leq C \vareps^\mu
\end{align}
for some $\mu>0$.\\
	
\noindent 		To estimate $F_2$, we use H\"older's inequality to have
		\[
		\|F_2\|_{L^{m_1}_t([T,\infty), L^{b_1}_x)} \leq \|F_2\|^{\frac{r_1}{b_1}}_{L^{m_1}_t([T,\infty), L^{r_1}_x)} \|F_2\|^{\frac{b_1-r_1}{b_1}}_{L^{m_1}_t([T,\infty), L^\infty_x)}.
		\]
		As $(m_1, r_1)$ is a Schr\"odinger admissible pair and
		\[
		F_2(t) = e^{i(t-T+\vareps^{-\sigma})\Delta} u(T-\vareps^{-\sigma}) - e^{it\Delta} u_0,
		\]
		we have
		\[
		\|F_2\|_{L^{m_1}_t([T,\infty), L^{r_1}_x)} \lesssim 1.
		\]
		On the other hand, by the dispersive estimates \eqref{dis-est} and Sobolev embedding, and noting that $m_1>2$, we have for all $t\geq T$, 
		\begin{align*}
		\|F_2(t)\|_{L^\infty_x} &\lesssim \int_0^{T-\vareps^{-\sigma}} (t-s)^{-\frac{3}{2}} \(\||u(s)|^{p-1} u(s)\|_{L^1_x} + \||u(s)|^{q-1} u(s)\|_{L^1_x} \) ds \\
		&=\int_0^{T-\vareps^{-\sigma}} (t-s)^{-\frac{3}{2}} \(\|u(s)\|^p_{L^p_x} + \|u(s)\|^q_{L^q_x} \) ds \\
		&\lesssim \int_0^{T-\vareps^{-\sigma}} (t-s)^{-\frac{3}{2}} \(\|u(s)\|^p_{H^1_x} + \|u(s)\|^q_{H^1_x} \) ds \\
		&\lesssim (t-T+\vareps^{-\sigma})^{-\frac{1}{2}}.
		\end{align*}
		It follows that
		\begin{align*}
		\|F_2\|_{L^{m_1}_t([T,\infty), L^\infty_x)} \lesssim \Big( \int_T^\infty (t-T+\vareps^{-\sigma})^{-\frac{m_1}{2}} dt\Big)^{\frac{1}{m_1}} \lesssim \vareps^{\sigma  \(\frac{1}{2}-\frac{1}{m_1}\)}.
		\end{align*}
		In particular, we get
		\begin{align} \label{est-F2}
		\|F_2\|_{L^{m_1}_t([T,\infty), L^{b_1}_x)} \lesssim \vareps^{\sigma\(\frac{1}{2}-\frac{1}{m_1}\) \frac{r_1-b_1}{r_1}}.
		\end{align}
		Collecting \eqref{est-line-part}, \eqref{est-F1}, \eqref{est-F2}, and
		choosing $\sigma>0$ sufficiently small, we prove \eqref{scat-cond-appl}. The proof is complete.
	\end{proof}
	
	\section{Energy Scattering}
This section is the bulk of the paper and contains the main novelty. Specifically, we will prove a coercivity property, see Lemma \ref{lem-A-omega-plus} below, and the interaction Morawetz estimates that will allow to prove the scattering for large data global solutions to \eqref{NLS}. 

	\subsection{A cutoff function}
	
	Let $\eta \in (0,1)$ be a small constant. Let $\chi$ be a smooth decreasing radial function satisfying
	\begin{align} \label{def-chi}
	\chi(x) =\chi(r)= \left\{
	\begin{array}{ccl}
	1 &\text{if}& r \leq 1-\eta, \\
	0 &\text{if}& r>1, 
	\end{array}
	\right. \quad |\chi'(r)| \lesssim \frac{1}{\eta}, \quad r=|x|.
	\end{align}
	For $R>0$ large, we define the functions
	\begin{align*} 
	\phi_R(x):= \frac{1}{\omega_3 R^3} \int \chi^2_R(x-z) \chi^2_R(z) dz 
	\end{align*}
	and
	\begin{align*} 
	\phi_{p,R}(x):&= \frac{1}{\omega_3 R^3} \int \chi^2_R(x-z) \chi^{p+1}_R(z) dz, \\
	\phi_{q,R}(x):&= \frac{1}{\omega_3 R^3} \int \chi^2_R(x-z) \chi^{q+1}_R(z) dz,
	\end{align*}
	where $\chi_R(z):=\chi(z/R)$, and $\omega_3$ is the volume of the unit ball in $\R^3$. We see that $\phi_R$, $\phi_{p,R}$, and $\phi_{q,R}$ are radial functions. We next define the radial function
	\begin{align} \label{def-psi}
	\psi_R(x) = \psi_R(r):= \frac{1}{r} \int_0^{r} \phi_R(\tau) d\tau, \quad r=|x|.
	\end{align}
	We collect some basic properties of $\phi_R, \phi_{p,R}, \phi_{q,R}$, and $\psi_R$ as follows (see \cite{Dinh-NDEA} for a proof.).
	\begin{lemma}\label{lem-proper-psi}
		We have 
		\begin{align*} 
		|\psi_R(x)| \lesssim \min \left\{1,\frac{R}{|x|} \right\},  \quad \partial_j \psi_R (x) =\frac{x_j}{|x|^2} \left(\phi_R(x) - \psi_R(x)\right), \quad j=1, 2, 3
		\end{align*}
		and
		\begin{align} \label{proper-psi-2}
		\psi_R(x) - \phi_R(x) \geq 0, \quad \phi_R(x)-\phi_{q,R}(x)\geq 0, \quad
		|\phi_R(x) - \phi_{p,R}(x)| \lesssim \eta
		\end{align}
		and
		\begin{align} \label{proper-psi-3}
		|\nabla \phi_R(x)| \lesssim \frac{1}{\eta R}, \quad |\psi_R(x) - \phi_R(x)| \lesssim \frac{1}{\eta}\min \left\{ \frac{|x|}{R}, \frac{R}{|x|}\right\}, \quad |\nabla \psi_R(x)| \lesssim \frac{1}{\eta}\min \left\{\frac{1}{R}, \frac{R}{|x|^2} \right\}
		\end{align}
		for all $x \in \R^3$.
	\end{lemma}
	
	\subsection{A coercivity property}\label{subsec:coer} The following is the essential new ingredient that will be exploited to prove suitable interaction Morawetz estimates. As already mentioned in the introduction, the scaling invariant equation with  one nonlinearity is treated by taking advantage of the refined Gagliardo-Nirenberg inequality \eqref{refined-GN}, that we cannot use in the present paper. 

	\begin{lemma} \label{lem-A-omega-plus} 
		Let $\frac{7}{3}<q<p<5$ and $\omega>0$. Then $\Ac^+_\omega$ is invariant under the flow of \eqref{NLS}, i.e., if $u_0 \in \Ac^+_\omega$, then $u(t) \in \Ac^+_\omega$ for all $t\in I_{\max}$. In particular, the solution to \eqref{NLS} with initial data $u_0$ exists globally in time. In addition, there exists $R_0>0$ sufficiently large such that for all $R\geq R_0$, all $z\in \R^3$, and all $t\in \R$,
		\begin{align} \label{est-G-u-chi-R}
		G(\chi_R(\cdot-z) u^\xi(t)) \geq \delta \|\nabla (\chi_R(\cdot-z) u^\xi(t))\|^2_{L^2},
		\end{align}
		where $u^\xi(t,x):=e^{ix\cdot \xi} u(t,x)$ with $\xi=\xi(t,z,R)$ and 
		\begin{align} \label{def-xi}
		\xi(t,z,R) := \left\{
		\begin{array}{cl}
		-\frac{\displaystyle{\int} \ima (\chi^2_R(x-z) \overline{u}(t,x) \nabla u(t,x)) dx}{\displaystyle{\int} \chi^2_R(x-z) |u(t,x)|^2 dx} &\text{ if } \mathlarger{\int} \chi^2_R(x-z) |u(t,x)|^2 dx \ne 0, \\
		0 &\text{ if } \mathlarger{\int} \chi^2_R(x-z) |u(t,x)|^2 dx =0,
		\end{array}
		\right.
		\end{align}
		and	$\chi_R(x)=\chi(x/R)$ with $\chi$ as in \eqref{def-chi}
	\end{lemma}
	
	\begin{proof}
		Let $u_0 \in \Ac^+_\omega$. We will show that $u(t) \in \Ac^+_\omega$ for all $t\in I_{\max}$. By the conservation of mass and energy, we have $S_\omega(u(t)) = S_\omega(u_0)<m_\omega$ for all $t\in I_{\max}$. Assume by contradiction that there exists $t_0$ such that $G(u(t_0))<0$. As $u:I_{\max} \to H^1$ is continuous, there exists $t_1$ such that $G(u(t_1))=0$. By the definition of $m_\omega$, we have $S_\omega(u(t_1)) \geq m_\omega$ which is a contradiction. Thus $G(u(t)) \geq 0$ for all $t\in I_{\max}$, namely $\Ac^+_\omega$ is invariant under the flow of \eqref{NLS}.
		
		As $G(u(t)) \geq 0$, we have
		\begin{align*}
		\frac{3q-7}{6(q-1)}&\|\nabla u(t)\|^2_{L^2} +\frac{p-q}{(p+1)(q-1)} \|u(t)\|^{p+1}_{L^{p+1}} +\frac{\omega}{2} \|u(t)\|^2_{L^2} \\
		&=I_\omega(u(t)) = S_\omega(u(t)) -\frac{2}{3(q-1)}G(u(t)) \leq S_\omega(u(t)) <m_\omega,  \quad \forall t\in I_{\max}.
		\end{align*}
	This shows that 
		\begin{align} \label{est-grad-u}
		\|\nabla u(t)\|^2_{L^2} \leq \frac{6(q-1)}{3q-7} m_\omega,\quad \forall t\in I_{\max}.
		\end{align}
		The blow-up alternative shows that the solution must exist globally in time.\\
			
		Next we prove \eqref{est-G-u-chi-R}. To this end, we observe that
		\begin{align} \label{m-omega-nu}
		I_\omega(u(t)) =S_\omega(u(t))-\frac{2}{3(q-1)} G(u(t)) \leq S_\omega(u(t)) =S_\omega(u_0)=m_\omega-\nu, \quad \forall t\in I_{\max},
		\end{align}
		for some $\nu>0$ as $S_\omega(u_0)<m_\omega$. Using 
		\begin{align}\label{proper-chi}
		\int |\nabla(\chi u)|^2 dx = \int \chi^2 |\nabla u|^2 dx -\int \chi \Delta \chi |u|^2 dx,
		\end{align}
		we have
		\begin{align*}
		\int |\nabla(\chi u^\xi)|^2 dx = |\xi|^2 \int \chi^2 |u|^2 dx + \int \chi^2 |\nabla u|^2 dx - \int \chi \Delta \chi |u|^2 dx + 2 \xi \cdot \int \ima(\chi^2 \overline{u} \nabla u) dx.
		\end{align*}
		By the definition of $\xi$, we have
		\begin{align*}
		I_\omega(\chi_R(\cdot-z) u^\xi(t)) &= \frac{3q-7}{6(q-1)} \Big( \int \chi^2_R(\cdot-z) |\nabla u(t)|^2 dx \\
		&\quad\quad\quad\quad\quad\quad\quad- \int \chi_R(\cdot-z) \Delta\chi_R (\cdot-z) |u(t)|^2 dx \Big) \\
		& - \frac{3q-7}{6(q-1)} \frac{\Big|\displaystyle{\int} \ima(\chi^2_R(\cdot-z) \overline{u}(t,x) \nabla u(t,x)) dx\Big|^2}{\displaystyle{\int} \chi^2_R(\cdot-z) |u(t,x)|^2 dx} \\
		& + \frac{p-q}{(p+1)(p-1)} \int \chi^{p+1}_R(\cdot-z) |u(t)|^{p+1} dx + \frac{\omega}{2} \int \chi^2_R(\cdot-z) |u(t)|^2 dx \\
		&\leq I_\omega(u(t)) + O(R^{-2}).
		\end{align*}
		Thus, by \eqref{m-omega-nu}, there exists $R_0>0$ sufficiently large such that 
		\begin{align} \label{I-omega-u}
		I_\omega(\chi_R(\cdot-z) u^\xi(t)) \leq  m_\omega -\frac{\nu}{2}, \quad \forall R\geq R_0, \, \forall z \in \R^3, \, \forall t\in \R.
		\end{align}
		We claim that 
		\begin{align} \label{claim-G}
		G(\chi_R(\cdot-z) u^\xi(t)) > 0, \quad  \forall R\geq R_0,\, \forall z\in \R^3,\, \forall t\in \R.
		\end{align} 
		Suppose now that there exists $R_1\geq R_0$, $z_1 \in \R^3$, $t_1\in \R$, and $\xi_1=\xi(t_1, z_1, R_1) \in \R^3$ such that $G(\chi_{R_1}(\cdot-z_1) u^{\xi_1}(t_1)) \leq 0$. Using the definition of $\tilde{m}_\omega$ (see \eqref{m-omega-tilde}) and the fact that $m_\omega=\tilde{m}_\omega$, we have
		\[
		I_\omega(\chi_{R_1}(\cdot-z_1) u^{\xi_1}(t_1)) \geq m_\omega
		\]
		which contradicts \eqref{I-omega-u}. 
		
		For $R\geq R_0, z\in \R^3$, $t\in \R$, and $\xi=\xi(t,z,R) \in \R^3$, we denote $\Theta:=\chi_R(x-z) u^\xi(t)$. We consider two cases.\medskip
		
\noindent 		{\bf Case 1.} Assume that
		\[
		4 \|\nabla \Theta\|^2_{L^2} - \frac{3(p-1)(3p+1)}{4(p+1)}\|\Theta\|^{p+1}_{L^{p+1}} + \frac{3(q-1)(3q+1)}{4(q+1)} \|\Theta\|^{q+1}_{L^{q+1}} \geq 0.
		\]
		Then we have
		\begin{align*}
		G(\Theta)&= \|\nabla \Theta\|^2_{L^2} -\frac{3(p-1)}{2(p+1)}\|\Theta\|^{p+1}_{L^{p+1}} +\frac{3(q-1)}{2(q+1)}\|\Theta\|^{q+1}_{L^{q+1}} \\
		&\geq \frac{3p-7}{3p+1}\|\nabla \Theta\|^2_{L^2} + \frac{9(q-1)(p-q)}{2(q+1)(3p+1)} \|\Theta\|^{q+1}_{L^{q+1}} \\
		&\geq \frac{3p-7}{3p+1}\|\nabla \Theta\|^2_{L^2}.
		\end{align*}
		This proves \eqref{est-G-u-chi-R}.
		\medskip
		
\noindent		{\bf Case 2.} We now assume that
		\begin{align} \label{case-2}
		4\|\nabla \Theta\|^2_{L^2} -\frac{3(p-1)(3p+1)}{4(p+1)} \|\Theta\|^{p+1}_{L^{p+1}} + \frac{3(q-1)(3q+1)}{4(q+1)} \|\Theta\|^{q+1}_{L^{q+1}} <0.
		\end{align}
		We first observe that as $I_\omega(\Theta)<m_\omega$ (see \eqref{I-omega-u}), an argument leading to \eqref{est-grad-u} yields
		\begin{align} \label{est-grad-g}
		\|\nabla \Theta\|_{L^2}^2 \leq \frac{6(q-1)}{3q-7} m_\omega.
		\end{align}
		Now set $f(\lambda):= S_\omega(\Theta_\lambda)$, with $\Theta_\lambda$ as in the rescaling \eqref{scaling}. We have
		\[
		f'(\lambda) = \lambda \|\nabla \Theta\|^2_{L^2} -\frac{3(p-1)}{2(p+1)} \lambda^{\frac{3(p-1)}{2}-1} \|\Theta\|^{p+1}_{L^{p+1}} + \frac{3(q-1)}{2(q+1)} \lambda^{\frac{3(q-1)}{2}-1} \|\Theta\|^{q+1}_{L^{q+1}} = \frac{G(\Theta_\lambda)}{\lambda}
		\]
		and
		\[
		\(\lambda f'(\lambda)\)' = 2\lambda \|\nabla \Theta\|^2_{L^2} -\frac{9(p-1)^2}{4(p+1)} \lambda^{\frac{3(p-1)}{2}-1} \|\Theta\|^{p+1}_{L^{p+1}} + \frac{9(q-1)^2}{4(q+1)} \lambda^{\frac{3(q-1)}{2}-1} \|\Theta\|^{q+1}_{L^{q+1}}.
		\]
		We write
		\begin{align*}
		\(\lambda f'(\lambda)\)' &= -2 f'(\lambda) + \lambda \Big( 4\|\nabla \Theta\|^2_{L^2} -\frac{3(p-1)(3p+1)}{4(p+1)} \lambda^{\frac{3p-7}{2}}\|\Theta\|^{p+1}_{L^{p+1}} \\
		&\quad \mathrel{\phantom{-2 h'(\lambda) + \lambda \Big( 4\|\nabla \Theta\|^2_{L^2}}} + \frac{3(q-1)(3q+1)}{4(q+1)} \lambda^{\frac{3q-7}{2}}\|\Theta\|^{q+1}_{L^{q+1}}  \Big) \\
		&=: -2f'(\lambda) + \lambda h(\lambda).
		\end{align*}
		Thanks to \eqref{case-2}, we have
		\begin{align*}
		h'(\lambda) &= \lambda^{\frac{3q-7}{2}} \left(-\frac{3(p-1)(3p+1)(3p-7)}{8(p+1)} \lambda^{\frac{3(p-q)}{2}} \|\Theta\|^{p+1}_{L^{p+1}} \right.\\
	&\quad\quad\quad\quad \quad\left.+ \frac{3(q-1)(3q+1)(3p-7)}{8(q+1)}\|\Theta\|^{q+1}_{L^{q+1}}\right) \\
		&\leq \lambda^{\frac{3q-7}{2}} \left(- \frac{3(p-1)(3p+1)(3p-7)}{2(p+1)} \lambda^{\frac{3(p-q)}{2}} \|\nabla \Theta\|^2_{L^2} \right.\\
		&\quad\quad\quad\quad \quad\left.+\frac{3(q-1)(3q+1)}{4(q+1)}\left(-\frac{3p-7}{2}\lambda^{\frac{3(p-q)}{2}} +\frac{3q-7}{2}\right)\|\Theta\|^{q+1}_{L^{q+1}}   \right) \\
		&\leq \lambda^{\frac{3q-7}{2}} \left(-\frac{3(p-1)(3p+1)(3p-7)}{2(p+1)}  \lambda^{\frac{3(p-q)}{2}} \|\nabla \Theta\|^2_{L^2}\right.\\
		 &\quad\quad\quad\quad\quad\left.- \frac{9(q-1)(3q+1)(p-q)}{8(q+1)} \|\Theta\|^{q+1}_{L^{q+1}}\right) <0
		\end{align*}
		for all $\lambda \geq 1$. This shows that $h(\lambda)\leq h(1)<0$ for all $\lambda\geq 1$.  In particular, we have
		\begin{align} \label{omega-plus}
		\(\lambda f'(\lambda)\)' \leq -2f'(\lambda), \quad \forall \lambda \geq 1.
		\end{align}
		As $G(\Theta)>0$ due to \eqref{claim-G}, Lemma \ref{lem-lamb-0} shows $G(\Theta_{\lambda_0})=0$ for some $\lambda_0>1$. Integrating \eqref{omega-plus} over $(1,\lambda_0)$, we get
		\begin{align*} 		
		G(\Theta) \geq 2(S_\omega(\Theta_{\lambda_0}) - S_\omega(\Theta)) = 2 \left(I_\omega(\Theta_{\lambda_0}) - I_\omega(\Theta) - \frac{2}{3(q-1)} G(\Theta) \right).
		\end{align*}
		It follows that
		\begin{align*}
		G(\Theta) &\geq \frac{6(q-1)}{3q+1} (I_\omega(\Theta_{\lambda_0}) - I_\omega(\Theta)) \\
		&\geq \frac{6(q-1)}{3q+1}(m_\omega -I_\omega(\Theta)) \\
		&\geq \frac{6(q-1)}{3q+1}\left(1-\frac{I_\omega(\Theta)}{m_\omega}\right) m_\omega \\
		&\geq \frac{(3q-7)\nu}{2(3q+1)m_\omega} \|\nabla \Theta\|^2_{L^2},
		\end{align*}
		where we have used \eqref{I-omega-u} and \eqref{est-grad-g} to get the last inequality. This also proves \eqref{est-G-u-chi-R}. 		
	\end{proof}

	\subsection{An interaction Morawetz estimate}
	
	We next define the interaction Morawetz action
	\begin{align} \label{defi-M-R}
	\Mca^{\otimes 2}_R(t):= 2\iint |u(t,y)|^2 \psi_R(x-y) (x-y) \cdot \ima ( \overline{u}(t,x) \nabla u(t,x)) dx dy,
	\end{align}
	where $\psi_R$ is as in \eqref{def-psi}.  We start by the following interaction Morawetz identity.

	\begin{lemma}
			Let $u$ be a solution to \eqref{NLS} satisfying
		\begin{align*} 
		\sup_{t\in [0,T^*)} \|u(t)\|_{H^1_x} \leq A
		\end{align*}
		for some constant $A>0$. Let $\Mca^{\otimes 2}_R(t)$ be as in \eqref{defi-M-R}. Then we have
		\begin{align} \label{est-M-R}
		\sup_{t\in [0,T^*)} |\Mca^{\otimes 2}_R(t)| \lesssim_A R.
		\end{align}
		Moreover, we have
		\begin{align}\notag
		\frac{d}{dt} \Mca^{\otimes 2}_R(t) &= - 4 \sum_{j,k}\iint \partial_j \left(\ima (\overline{u}(t,y) \partial_j u(t,y)) \right) \\
		& \quad\quad\quad\quad\quad\quad\times \psi_R(x-y) (x_k-y_k) \ima (\overline{u}(t,x) \partial_k u(t,x)) dx dy \label{term-1}\\\notag
		& - 4 \sum_{j,k} \iint |u(t,y)|^2 \psi_R(x-y) (x_j-y_j) \\
		&\quad\quad\quad\quad\quad\times  \partial_k \left( \rea (\partial_j u(t,x) \partial_k \overline{u}(t,x)) \right) dx dy \label{term-2} \\
		& + \iint |u(t,y)|^2 \psi_R(x-y) (x-y) \cdot \nabla \Delta(|u(t,x)|^2) dx dy \label{term-3} \\
		& + \frac{2(p-1)}{p+1} \iint |u(t,y)|^2 \psi_R(x-y) (x-y) \cdot \nabla (|u(t,x)|^{p+1}) dx dy \label{term-4} \\
		& - \frac{2(q-1)}{q+1} \iint |u(t,y)|^2 \psi_R(x-y) (x-y) \cdot \nabla (|u(t,x)|^{q+1}) dx dy \label{term-5}
		\end{align}
		for all $t\in [0,T^*)$.
	\end{lemma}
	
	\begin{proof}
		By the support property of $\chi$, we have $\phi_R(\tau) =0$ for all $|\tau|\geq 2R$. Thus we get
		\begin{align} \label{est-psi}
		\psi_R(x) |x| \leq  \int_0^{2R} \phi_R(\tau) d\tau =\frac{R}{\omega_3} \int_0^2\int \chi^2(x-z) \chi^2(z) dz d\tau =CR 
		\end{align}
		for some constant $C>0$ independent of $R$. The estimate \eqref{est-M-R} follows directly from H\"older's inequality and \eqref{est-psi}. The identities \eqref{term-1}--\eqref{term-5} follow from a direct computation using 
		\begin{align*}
		\partial_t (|u|^2) = -2\sum_{j} \partial_j(\ima(\overline{u} \partial_j u))
		\end{align*}
		and 
		\begin{align*}
		\partial_t (\ima(\overline{u}\partial_j u)) &= - \sum_k \partial_k \left(2 \rea(\partial_j u \partial_k \overline{u}) - \frac{1}{2}\delta_{jk} \Delta (|u|^2)\right) \\
		&+ \frac{p-1}{p+1} \partial_j (|u|^{p+1}) - \frac{q-1}{q+1}\partial_j(|u|^{q+1}),
		\end{align*}
		for $j=1,2,3$, where $\delta_{jk}$ is the Kronecker symbol.
	\end{proof}
We are now in position to prove our interaction Morawetz estimates, that jointly to the scattering criterion of the previous section will yield to the main result of the paper. 	The coercivity result in Lemma \ref{lem-A-omega-plus} is essential for the proof of the estimates below.
	\begin{proposition}\label{prop-inter-mora-est}
		Let $\frac{7}{3}<q<p<5$ and $\omega>0$. Let $u_0 \in \Ac^+_\omega$ and $u(t)$ be the corresponding solution to \eqref{NLS}. Define $\Mca^{\otimes 2}_R(t)$ as in \eqref{defi-M-R}. Then for $\vareps>0$ sufficiently small, there exist $T_0 = T_0(\vareps)$, $J=J(\vareps)$, $R_0= R_0(\vareps, u_0)$ sufficiently large, and $\eta=\eta(\vareps)>0$ sufficiently small such that for any $a \in \R$,
		\begin{align} \label{inter-mora-est}
		\frac{1}{JT_0} \int_a^{a+T_0}\int_{R_0}^{R_0e^J} \frac{1}{R^3} \iiint \left|\chi_R(y-z)  u(t,y)\right|^2 | \nabla ( \chi_R(x-z)u^\xi(t,x) )|^2 dx dy dz \frac{dR}{R} dt \lesssim \vareps,
		\end{align}
		where $\chi_R(x)=\chi(x/R)$ with $\chi$ as in \eqref{def-chi} and $u^\xi(t,x)=e^{ix\cdot \xi}u(t,x)$ with some $\xi= \xi(t,z,R) \in \R^3$.
	\end{proposition}
	
	\begin{proof}
		Denote
		\[
		P_{jk}(x):= \delta_{jk} - \frac{x_jx_k}{|x|^2}.
		\]
		As 
		\[
		\partial_j(\psi_R x_k) = \delta_{jk}\psi_R +\frac{x_jx_k}{|x|^2} (\phi_R-\psi_R),
		\]
		the integration by parts yields
		\begin{align}
		\eqref{term-1} &= 4\sum_{j,k} \iint \ima(\overline{u}(t,y) \partial_j u(t,y)) \partial^y_j(\psi_R(x-y) (x_k-y_k)) \ima(\overline{u}(t,x) \partial_ku(t,x)) dxdy \nonumber \\
		&=-4 \sum_{j,k} \iint \ima(\overline{u}(t,y) \partial_j u(t,y)) \delta_{jk} \phi_R(x-y) \ima(\overline{u}(t,x) \partial_k u(t,x)) dxdy \label{term-11}\\\notag
		& -4 \sum_{j,k} \iint \ima(\overline{u}(t,y) \partial_ju(t,y)) P_{jk}(x-y) \\
		&\quad\quad\quad\quad\quad\times(\psi_R-\phi_R)(x-y) \ima(\overline{u}(t,x) \partial_k u(t,x)) dxdy, \label{term-12}
		\end{align}
		where $\partial^y_j$ is $\partial_j$ with respect to the $y$-variable. Similarly, we have
		\begin{align}
		\eqref{term-2} &=4\sum_{j,k} \iint |u(t,y)|^2 \partial^x_k(\psi_R(x-y) (x_j-y_j)) \rea(\partial_j u(t,x) \partial_k \overline{u}(t,x)) dxdy \nonumber \\
		&=4\sum_{j,k} \iint |u(t,y)|^2 \delta_{jk}\phi_R(x-y) \rea(\partial_j u(t,x) \partial_k \overline{u}(t,x)) dxdy \label{term-21} \\
		&- 4\sum_{j,k} \iint |u(t,y)|^2 P_{jk}(x-y) (\psi_R-\phi_R)(x-y) \rea(\partial_j u(t,x) \partial_k \overline{u}(t,x)) dxdy, \label{term-22}
		\end{align}
		where $\partial^x_k$ is $\partial_k$ with respect to the $x$-variable. We have
		\begin{align*}
		\eqref{term-12}+\eqref{term-22} &=4 \iint |u(t,y)|^2 |\nnabla_y u(t,x)|^2 (\psi_R-\phi_R)(x-y) dxdy  \\
		& -4\iint \ima(\overline{u}(t,y) \nnabla_x u(t,y)) \cdot \ima(\overline{u}(t,x) \nnabla_y u(t,x)) (\psi_R-\phi_R)(x-y) dxdy,
		\end{align*}
		where
		\begin{align*}
		\nnabla_y u(t,x) := \nabla u(t,x) - \frac{x-y}{|x-y|} \left( \frac{x-y}{|x-y|} \nabla u(t,x)\right)
		\end{align*}
		is the angular derivative centered at $y$, and similarly for $\nnabla_x u(t,y)$. As $\psi_R-\phi_R\geq 0$, the Cauchy-Schwarz inequality yields
		\[
		\eqref{term-12}+\eqref{term-22} \geq 0.
		\]
		Next, using the fact that
		\[
		\phi_R(x-y) =\frac{1}{\omega_3R^3}\int \chi^2_R(x-y-z) \chi^2_R(z) dz = \frac{1}{\omega_3R^3} \int \chi^2_R(x-z) \chi^2_R(y-z) dz,
		\]
		we have
		\begin{align*}
		\eqref{term-11}+\eqref{term-21} &= 4\iint \phi_R(x-y) \left(|u(t,y)|^2|\nabla u(t,x)|^2\right. \\
		&\left.\quad\quad\quad\quad\quad\quad\quad\quad- \ima(\overline{u}(t,y) \nabla u(t,y)) \cdot \ima(\overline{u}(t,x) \nabla u(t,x))\right) dxdy \\
		&= \frac{4}{\omega_3R^3}\iiint \chi^2_R(x-z) \chi^2_R(y-z) \left(|u(t,y)|^2|\nabla u(t,x)|^2\right.\\ 
		&\left.\quad\quad\quad\quad\quad\quad\quad\quad- \ima(\overline{u}(t,y) \nabla u(t,y)) \cdot \ima(\overline{u}(t,x) \nabla u(t,x)) \right) dxdydz.
		\end{align*}
		For $z\in \R^3$ fixed, we consider the quantity defined by
		\begin{align*}
		\iint \chi^2_R(x-z) \chi^2_R(y-z) &\left(|u(t,y)|^2 |\nabla u(t,x)|^2\right.\\
		&\left.\quad-\ima(\overline{u}(t,y) \nabla u(t,y)) \cdot \ima(\overline{u}(t,x) \nabla u(t,x)) \right) dxdy.
		\end{align*}
		It is not hard to see that the above quantity is invariant under the Galilean transformation $u(t,x) \mapsto u^\xi(t,x)$ for all $\xi \in \R^3$ due to the symmetry of $\chi$. We will choose a suitable $\xi\in \R^3$ such that
		\[
		\int \ima(\chi^2_R(x-z) \overline{u}^\xi(t,x) \nabla (u^\xi(t,x))) dx =0.
		\]
		Specifically, we select $\xi$ as follows (compare with \eqref{def-xi}): 
		\[
		\xi =\xi(t,z,R)= - \frac{\displaystyle{\int} \ima(\chi^2_R(x-z) \overline{u}(t,x) \nabla u(t,x)) dx }{\displaystyle{\int} \chi^2_R(x-z) |u(t,x)|^2 dx},
		\]
		provided that the denominator is non-zero (otherwise $\xi =0$ suffices). With this $\xi$, we get
		\begin{align*} 
		\eqref{term-11}+\eqref{term-21} = \frac{4}{\omega_3 R^3} \iiint |\chi_R(y-z) u(t,y)|^2 |\chi_R(x-z) \nabla (u^\xi(t,x))|^2 dxdydz. 
		\end{align*}
		By integration by parts twice and using 
		\[
		\sum_j \partial_j (\psi_R x_j) = 3\psi_R +\sum_j x_j \partial_j \psi_R = 3\phi_R + 2(\psi_R-\phi_R),
		\]
		we have
		\begin{align}
		\eqref{term-3} &= \sum_{j,k} \iint |u(t,y)|^2 \psi_R(x-y) (x_j-y_j) \partial_j \partial^2_k(|u(t,x)|^2) dxdy \nonumber \\
		&=-\sum_{j,k} \iint |u(t,y)|^2 \partial^x_j(\psi_R(x-y)(x_j-y_j)) \partial^2_k(|u(t,x)|^2) dxdy \nonumber \\
		&=\sum_{k} \iint |u(t,y)|^2 \partial^x_k(3\phi_R(x-y) + 2(\psi_R-\phi_R)(x-y)) \partial_k(|u(t,x)|^2) dxdy. \nonumber
		\end{align}
		We also have
		\begin{align}
		\eqref{term-4} &= -\frac{2(p-1)}{p+1} \sum_j \iint |u(t,y)|^2 \partial^x_j(\psi_R(x-y) (x_j-y_j)) |u(t,x)|^{p+1} dxdy \nonumber \\
		&=-\frac{6(p-1)}{p+1}\iint |u(t,y)|^2 \phi_{p,R}(x-y)|u(t,x)|^{p+1} dxdy \label{term-4-1}\\
		& -\frac{6(p-1)}{p+1} \iint |u(t,y)|^2 (\phi_R-\phi_{p,R})(x-y) |u(t,x)|^{p+1} dxdy \nonumber \\
		& -\frac{4(p-1)}{p+1}\iint |u(t,y)|^2 (\psi_R-\phi_R)(x-y) |u(t,x)|^{p+1} dxdy.\nonumber 
		\end{align}
		We can rewrite \eqref{term-4-1} as
		\begin{align*}
		\eqref{term-4-1} = -\frac{6(p-1)}{(p+1)\omega_3 R^3} \iiint |\chi_R(y-z) u(t,y)|^2 |\chi_R(x-z) u(t,x)|^{p+1} dxdydz. 
		\end{align*}
		Similarly, we have
		\begin{align*}
		\eqref{term-5} &=\frac{6(q-1)}{(q+1)\omega_3 R^3}\iiint |\chi_R(y-z)u(t,y)|^2 |\chi_R(x-z)u(t,x)|^{q+1} dxdydz \\
		& +\frac{6(q-1)}{q+1} \iint |u(t,y)|^2 (\phi_R-\phi_{q,R})(x-y) |u(t,x)|^{q+1} dxdy\\
		& +\frac{4(q-1)}{q+1}\iint |u(t,y)|^2 (\psi_R-\phi_R)(x-y) |u(t,x)|^{q+1} dxdy.
		\end{align*}
		Collecting the above identities, we obtain
		\begin{align*}
		\frac{d}{dt} \Mca^{\otimes 2}_R(t) &\geq \frac{4}{\omega_3R^3} \iiint |\chi_R(y-z) u(t,y)|^2 |\chi_R(x-z) \nabla (u^\xi(t,x))|^2 dxdydz \\
		&  + \iint |u(t,y)|^2 \nabla(3\phi_R(x-y) + 2(\psi_R-\phi_R)(x-y)) \cdot \nabla(|u(t,x)|^2) dxdy \\
		& - \frac{6(p-1)}{(p+1)\omega_3 R^3} \iiint |\chi_R(y-z) u(t,y)|^2 |\chi_R(x-z) u(t,x)|^{p+1} dxdydz \\
		& -\frac{6(p-1)}{p+1} \iint |u(t,y)|^2 (\phi_R-\phi_{p,R})(x-y) |u(t,x)|^{p+1} dxdy \\
		& -\frac{4(p-1)}{p+1}\iint |u(t,y)|^2 (\psi_R-\phi_R)(x-y) |u(t,x)|^{p+1} dxdy \\
		& + \frac{6(q-1)}{(q+1)\omega_3 R^3}\iiint |\chi_R(y-z)u(t,y)|^2 |\chi_R(x-z)u(t,x)|^{q+1} dxdydz \\
		& +\frac{6(q-1)}{q+1} \iint |u(t,y)|^2 (\phi_R-\phi_{q,R})(x-y) |u(t,x)|^{q+1} dxdy \\
		& +\frac{4(q-1)}{q+1}\iint |u(t,y)|^2 (\psi_R-\phi_R)(x-y) |u(t,x)|^{q+1} dxdy.
		\end{align*}
		As $\psi_R -\phi_R\geq 0$ and $\phi_R -\phi_{q,R}\geq 0$, we get
		\begin{align*}
		&\frac{4}{\omega_3R^3} \iiint |\chi_R(y-z) u(t,y)|^2 \Big(|\chi_R(x-z)\nabla(u^\xi(t,x))|^2 \\
		&\quad\quad\quad\quad\quad- \frac{3(p-1)}{2(p+1)} |\chi_R(x-z) u(t,x)|^{p+1} + \frac{3(q-1)}{2(q+1)} |\chi_R(x-z) u(t,x)|^{q+1}\Big) dxdydz \\
		&\leq \frac{d}{dt}\Mca^{\otimes 2}_R(t) - \iint |u(t,y)|^2 \nabla\left(3\phi_R(x-y) + 2(\psi_R-\phi_R)(x-y)\right)\cdot \nabla (|u(t,x)|^2) dxdy \\
		& +\frac{6(p-1)}{p+1} \iint |u(t,y)|^2 (\phi_R-\phi_{p,R})(x-y) |u(t,x)|^{p+1} dxdy \\
		& +\frac{4(p-1)}{p+1} \iint |u(t,y)|^2 (\psi_R-\phi_R)(x-y) |u(t,x)|^{p+1} dxdy.
		\end{align*}
		By \eqref{est-M-R}, we see that
		\begin{align}
		\left| \frac{1}{JT_0} \int_a^{a+T_0} \int_{R_0}^{R_0 e^J} \frac{d}{dt} \Mca^{\otimes 2}_R(t) \frac{dR}{R} dt \right| &\leq \frac{1}{JT_0} \int_{R_0}^{R_0e^J} \sup_{t\in [a,a+T_0]} |\Mca^{\otimes 2}_R(t)| \frac{dR}{R} \nonumber \\
		&\lesssim \frac{1}{JT_0} \int_{R_0}^{R_0 e^J} dR \lesssim \frac{R_0e^J}{JT_0}. \label{est-1}
		\end{align}
		As $|\nabla \phi_R(x)| \lesssim \frac{1}{\eta R}$ and $\sup_{t\in \R}\|u(t)\|_{H^1}<\infty$, we see that
		\begin{align}
		\Big| \frac{1}{JT_0} \int_a^{a+T_0} \int_{R_0}^{R_0e^J} &\iint |u(t,y)|^2 \nabla \phi_R(x-y) \cdot \nabla (|u(t,x)|^2) dx dy \frac{dR}{R} dt \Big| \nonumber \\
		&\lesssim \frac{1}{\eta JT_0} \int_a^{a+T_0} \int_{R_0}^{R_0e^J} \|u(t)\|^3_{L^2} \|\nabla u(t)\|_{L^2} \frac{dR}{R^2} dt \nonumber \\
		&\lesssim \frac{1}{\eta JT_0} \int_a^{a+T_0} \int_{R_0}^{R_0e^J} \frac{dR}{R^2} dt \nonumber \\
		&\lesssim \frac{1}{\eta JR_0}. \label{est-2}
		\end{align}
		Similarly, as $|\nabla(\psi_R-\phi_R)(x)| \lesssim \frac{1}{\eta }\min \left\{\frac{1}{R},\frac{R}{|x|^2} \right\} <\frac{1}{\eta R}$, we have
		\begin{align}
		\Big| \frac{1}{JT_0} \int_a^{a+T_0} \int_{R_0}^{R_0e^J}  \iint |u(t,y)|^2 \nabla (\psi_R-\phi_R)(x-y) \cdot \nabla(|u(t,x)|^2) dx dy \frac{dR}{R} dt\Big| \lesssim \frac{1}{\eta JR_0}. \label{est-3}
		\end{align}
		Using \eqref{proper-psi-3}, we see that
		\begin{align}
		\Big|&\frac{1}{JT_0} \int_a^{a+T_0} \int_{R_0}^{R_0e^J} \iint |u(t,y)|^2 (\psi_R-\phi_R)(x-y) |u(t,x)|^{p+1} dxdy \frac{dR}{R} dt \Big| \nonumber \\
		&\lesssim \frac{1}{\eta JT_0} \int_a^{a+T_0} \int_{R_0}^{R_0e^J} \iint |u(t,y)|^2 \min \left\{ \frac{|x-y|}{R}, \frac{R}{|x-y|}\right\} |u(t,x)|^{p+1} dx dy \frac{dR}{R} dt \nonumber \\
		&\lesssim \frac{1}{\eta JT_0} \int_a^{a+T_0} \iint |u(t,y)|^2 |u(t,x)|^{p+1} \left( \int_{R_0}^{R_0e^J} \min \left\{ \frac{|x-y|}{R}, \frac{R}{|x-y|}\right\} \frac{dR}{R} \right) dx dy dt \nonumber \\
		&\lesssim \frac{1}{\eta J}. \label{est-4}
		\end{align}
		Here we have used the fact that $\sup_{t\in \R} \|u(t)\|_{H^1}<\infty$ and 
		\begin{align*} 
		\int_0^\infty \min \left\{ \frac{|x-y|}{R}, \frac{R}{|x-y|}\right\} \frac{dR}{R} \lesssim 1.
		\end{align*}
		Using \eqref{proper-psi-2}, we have
		\begin{align}
		\Big|\frac{1}{JT_0} \int_a^{a+T_0} \int_{R_0}^{R_0e^J} &\iint |u(t,y)|^2 (\phi_R-\phi_{p,R})(x-y) |u(t,x)|^{p+1} dx dy \frac{dR}{R} dt \Big| \nonumber \\
		&\lesssim \frac{1}{JT_0} \int_a^{a+T_0} \int_{R_0}^{R_0e^J} \eta \frac{dR}{R} dt \lesssim \eta. \label{est-5}
		\end{align}
		
\noindent		By glueing up together \eqref{est-1}, \eqref{est-2}, \eqref{est-3}, \eqref{est-4}, and \eqref{est-5}, we obtain 
		\begin{align}
		\Big|&\frac{1}{JT_0} \int_a^{a+T_0} \int_{R_0}^{R_0e^J} \frac{1}{R^3} \iiint |\chi_R(y-z) u(t,y)|^2 \Big( |\chi_R(x-z)\nabla(u^\xi(t,x))|^2  \nonumber \\
		&\quad  - \frac{3(p-1)}{2(p+1)} |\chi_R(x-z) u(t,x)|^{p+1} + \frac{3(q-1)}{2(q+1)} |\chi_R(x-z) u(t,x)|^{q+1} \Big) dxdydz \frac{dR}{R} dt \Big| \nonumber \\
		&\quad \lesssim \eta +\frac{R_0e^J}{\eta JT_0} + \frac{1}{\eta J} + \frac{1}{\eta JR_0}. \label{est-lhs}
		\end{align}
		Now, for fixed $z, \xi \in \R^3$, we have from \eqref{proper-chi} that
		\[
		\int |\chi_R(x-z) \nabla (u^\xi(t,x))|^2 dx =\int |\nabla[ \chi_R(x-z) u^{\xi}(t,x)]|^2 dx + O(R^{-2} \|u(t)\|^2_{L^2}).
		\]
		From the conservation of mass and \eqref{est-G-u-chi-R} that for $R\geq R_0$ with $R_0$ sufficiently large,
		\begin{align*}
		\int\Big( |\chi_R(x-z) &\nabla (u^{\xi}(t,x))|^2- \frac{3(p-1)}{2(p+1)} |\chi_R(x-z) u(t,x)|^{p+1} \\
		&\quad\quad\quad\quad\quad\quad\,+ \frac{3(q-1)}{2(q+1)} |\chi_R(x-z) u(t,x)|^{q+1} \Big)dx  \\
		&= G(\chi_R(\cdot-z) u^\xi(t)) + O(R^{-2}) \\
		&\geq \delta \|\nabla(\chi_R(\cdot-z) u^{\xi}(t))\|^2_{L^2} + O(R^{-2}).
		\end{align*}
		The term $O(R^{-2})$ can be treated analogously to \eqref{est-2}. We thus infer from \eqref{est-lhs} that
		\begin{align}
		\Big| \frac{1}{JT_0} \int_a^{a+T_0} &\int_{R_0}^{R_0e^J} \frac{1}{R^3} \iiint |\chi_R(y-z) u(t,y)|^2 |\nabla(\chi_R(x-z) u^\xi(t,x))|^2 dxdydz \frac{dR}{R} dt \Big| \nonumber \\
		&\lesssim \eta +\frac{R_0e^J}{\eta JT_0} + \frac{1}{\eta J} + \frac{1}{\eta JR_0}. \label{est-eta}
		\end{align}
		This proves \eqref{inter-mora-est} by taking $\eta=\vareps, J=\vareps^{-3}$, $R_0=\vareps^{-1}$ and $T_0=e^{\vareps^{-3}}$. The proof is complete.
	\end{proof}

	\subsection{Proof of the main result}
	We can now proceed with the proof of the main result. 
	\begin{proof}[Proof of Theorem \ref{theo-scat}]
		The global existence is proved in Lemma \ref{lem-A-omega-plus}. It remains to prove the scattering. We only consider the positive times since the one for negative times is similar. Our purpose is to check the scattering criteria given in Proposition \ref{prop-scat-crite}. To this end, we fix $a \in \R$ and let $\vareps>0$ sufficiently small and $T_0>0$ sufficiently large to be determined later. We will show that there exists $t_0 \in (a,a+T_0)$ such that $[t_0-\vareps^{-\sigma},t_0] \subset (a,a+T_0)$ and
		\begin{align} \label{scat-crite-cond}
		\|u\|_{L^{m_1}_t([t_0-\vareps^{-\sigma},t_0], L^{b_1}_x)} \lesssim \vareps^\mu
		\end{align}
		for some $\sigma, \mu>0$ to be determined later. By \eqref{inter-mora-est}, there exist $T_0 =T_0(\vareps), J=J(\vareps)$, $R_0=R_0(\vareps,u_0)$ and $\eta=\eta(\vareps)$ such that
		\[
		\frac{1}{JT_0} \int_a^{a+T_0} \int_{R_0}^{R_0e^J} \frac{1}{R^3} \iiint |\chi_R(y-z) u(t,y)|^2 |\nabla(\chi_R(x-z) u^{\xi}(t,x))|^2 dxdydz \frac{dR}{R} dt \lesssim \vareps.
		\]
		It follows that there exists $R_1 \in [R_0,R_0e^J]$ such that
		\[
		\frac{1}{T_0} \int_a^{a+T_0} \frac{1}{R_1^3}\iiint |\chi_{R_1}(y-z)u(t,y)|^2 |\nabla(\chi_{R_1}(x-z) u^{\xi}(t,x))|^2 dx dydz dt \lesssim \vareps
		\]
		hence
		\[
		\frac{1}{T_0} \int_a^{a+T_0} \frac{1}{R_1^3} \int \|\chi_{R_1}(\cdot-z) u(t)\|^2_{L^2_x} \|\nabla(\chi_{R_1}(\cdot-z) u^{\xi}(t))\|^2_{L^2_x} dz dt \lesssim \vareps.
		\]
		By the change of variable $z=\frac{R_1}{4}(w+\theta)$ with $w\in \Z^3$ and $\theta \in [0,1]^3$, we deduce that there exists $\theta_1 \in [0,1]^3$ such that
		\begin{align*}
		\frac{1}{T_0} \int_a^{a+T_0} \sum_{w\in \Z^3} \Big\|\chi_{R_1}\left(\cdot-\frac{R_1}{4}(w+\theta_1)\right) u(t)\Big\|^2_{L^2_x} \Big\|\nabla \left( \chi_{R_1}\left((\cdot - \frac{R_1}{4}(w+\theta_1)\right) u^{\xi}(t)\right) \Big\|^2_{L^2_x} dt\\ \lesssim \vareps.
		\end{align*}
		Let $\sigma>0$ to be chosen later. By dividing the interval $\left[a+\frac{T_0}{2}, a+\frac{3T_0}{4}\right]$ into $T_0\vareps^{\sigma}$ intervals of length $\vareps^{-\sigma}$, we infer that there exists $t_0 \in \left[a+\frac{T_0}{2},a+\frac{3T_0}{4}\right]$ such that $[t_0-\vareps^{-\sigma},t_0] \subset (a,a+T_0)$  and
		\[
		\int_{t_0-\vareps^{-\sigma}}^{t_0} \sum_{w\in \Z^3} \Big\|\chi_{R_1} \Big(\cdot -\frac{R_1}{4}(w+\theta_1)\Big)u(t)\Big\|^2_{L^2_x} \Big\| \nabla \left( \chi_{R_1}\Big(\cdot-\frac{R_1}{4}(w+\theta_1)\Big) u^{\xi}(t)\right)\Big\|^2_{L^2_x} dt \lesssim \vareps^{1-\sigma}.
		\]
		This, together with the Gagliardo-Nirenberg inequality
		\[
		\|u\|^4_{L^3_x} \lesssim \|u\|^2_{L^2_x} \|\nabla u^{\xi}\|^2_{L^2_x},
		\]
		implies that
		\begin{align} \label{est-t0-1}
		\int_{t_0 -\vareps^{-\sigma}}^{t_0} \sum_{w\in \Z^3} \Big\| \chi_{R_1} \Big( \cdot-\frac{R_1}{4}(w+\theta_1)\Big) u(t)\Big\|^4_{L^3_x} dt \lesssim \vareps^{1-\sigma}.
		\end{align}
		On the other hand, by H\"older's inequality, Cauchy-Schwarz inequality, and Sobolev embedding, we have
		\begin{align}
		\sum_{w\in \Z^3} &\Big\| \chi_{R_1} \Big( \cdot-\frac{R_1}{4}(w+\theta_1) \Big) u(t)\Big\|^2_{L^3_x} \nonumber  \\
		&\leq \sum_{w\in \Z^3} \Big\|\chi_{R_1}\Big(\cdot-\frac{R_1}{4}(w+\theta_1)\Big) u(t)\Big\|_{L^2_x} \Big\|\chi_{R_1}\Big(\cdot-\frac{R_1}{4}(w+\theta_1)\Big) u(t)\Big\|_{L^6_x} \nonumber \\
		&\leq \Big(\sum_{w\in \Z^3} \Big\|\chi_{R_1}\Big(\cdot-\frac{R_1}{4}(w+\theta_1)\Big) u(t)\Big\|^2_{L^2_x}  \Big)^{1/2} \Big( \sum_{w\in \Z^3} \Big\|\chi_{R_1}\Big(\cdot-\frac{R_1}{4}(w+\theta_1)\Big) u(t)\Big\|^2_{L^6_x}\Big)^{1/2} \nonumber \\
		&\lesssim \|u(t)\|_{L^2_x} \|u(t)\|_{H^1_x} \lesssim 1. \label{est-t0-2}
		\end{align}
		Here we have used the following estimate to get the last line:
		\begin{align*}
		\sum_{w\in \Z^3} &\Big\| \chi_{R_1} \Big( \cdot-\frac{R_1}{4}(w+\theta_1)\Big) u(t)\Big\|^2_{L^6_x} \\
		&\lesssim \sum_{w \in \Z^3} \Big\|\chi_{R_1} \Big(\cdot - \frac{R_1}{4} (w+\theta_1)\Big) \nabla u(t)\Big\|^2_{L^2_x} + \frac{1}{R_1^2} \Big\|(\nabla \chi)_{R_1}\Big(\cdot-\frac{R_1}{4}(w+\theta_1)\Big) u(t)\Big\|^2_{L^2_x} \\
		&\lesssim \|\nabla u(t)\|^2_{L^2_x} + \frac{1}{R_1^2 \eta^2} \|u(t)\|^2_{L^2_x} \lesssim \|u(t)\|^2_{H^1_x}
		\end{align*}
		as $R_1>R_0 = \vareps^{-1}=\eta^{-1}$ (see after \eqref{est-eta}). Note that $|\nabla \chi| \lesssim \eta^{-1}$ by the choice of $\chi$. Combining \eqref{est-t0-1} and \eqref{est-t0-2}, we get from the property of $\chi_{R_1}$, in conjunction with the H\"older and the Cauchy-Schwarz inequalities, that
		\begin{align*}
		\|u\|^3_{L^3_{t,x}([t_0-\vareps^{-\sigma},t_0] \times \R^3)} &\lesssim \int_{t_0-\vareps^{-\sigma}}^{t_0} \sum_{w\in \Z^3} \Big\| \chi_{R_1} \Big( \cdot-\frac{R_1}{4} (w+\theta_1) \Big) u(t)\Big\|^3_{L^3_x} dt \\
		&\lesssim \int_{t_0-\vareps^{-\sigma}}^{t_0} \Big(\sum_{w\in \Z^3} \Big\|\chi_{R_1} \Big(\cdot- \frac{R_1}{4}(w+\theta_1) \Big) u(t)\Big\|^4_{L^3_x} \Big)^{\frac{1}{2}} \\
		&\mathrel{\phantom{\lesssim \int_{t_0-\vareps^{-\sigma}}^{t_0}}} \times \Big( \sum_{w\in \Z^3} \Big\|\chi_{R_1} \Big(\cdot- \frac{R_1}{4}(w+\theta_1) \Big) u(t)\Big\|^2_{L^3_x}\Big)^{\frac{1}{2}} dt  \\
		&\lesssim \Big(\int_{t_0-\vareps^{-\sigma}}^{t_0} \sum_{w\in \Z^3} \Big\| \chi_{R_1} \Big( \cdot- \frac{R_1}{4} (w+\theta_1)\Big) u(t)\Big\|^4_{L^3_x} dt \Big)^{\frac{1}{2}} \\
		&\mathrel{\phantom{\lesssim \int_{t_0-\vareps^{-\sigma}}^{t_0}}} \times \Big(\int_{t_0-\vareps^{-\sigma}}^{t_0} \sum_{w\in \Z^3} \Big\| \chi_{R_1} \Big( \cdot- \frac{R_1}{4} (w+\theta_1)\Big) u(t)\Big\|^2_{L^3_x} dt \Big)^{\frac{1}{2}} \\
		&\lesssim \vareps^{\frac{1-\sigma}{2}} \vareps^{-\frac{\sigma}{2}} = \vareps^{\frac{1}{2} -\sigma},
		\end{align*}
		which implies that
		\begin{align} \label{est-t0}
		\|u\|_{L^3_{t,x}([t_0-\vareps^{-\sigma},t_0] \times \R^3)} \lesssim \vareps^{\frac{1}{3}\left(\frac{1}{2}-\sigma \right)}.
		\end{align}
		Let $\theta \in (0,1)$ to be chosen shortly. We define $(\gamma, \rho)$ by
		\[
		\frac{1}{m_1}=\frac{\theta}{3}+\frac{1-\theta}{\gamma}, \quad \frac{1}{b_1}=\frac{\theta}{3}+\frac{1-\theta}{\rho}. 
		\]
		Pick $\beta, s>0$ such that
		\[
		\frac{1}{\rho}=\frac{1}{\beta}-\frac{s}{3}, \quad \frac{2}{\gamma}+\frac{3}{\beta}=\frac{3}{2}.
		\]
		We readily check that
		\[
		\frac{2}{m_1}+\frac{3}{b_1} =\frac{2}{p-1} = \frac{5\theta}{3}+(1-\theta) \left(\frac{3}{2}-s\right).
		\]
		In particular, 
		\[
		s=\frac{3}{2} -\frac{1}{1-\theta}\left(\frac{2}{p-1} - \frac{5\theta}{3}\right).
		\]
		As $7/3<p<5$, we can take $\theta>0$ sufficiently small so that $0<s<1$. In particular, $(\gamma, \beta)$ is a Schr\"odinger admissible pair. By H\"older's inequality, Sobolev embedding, and \eqref{est-t0}, we have
		\begin{align*}
		\|u\|_{L^{m_1}_t([t_0-\vareps^{-\sigma}, t_0], L^{b_1}_x)} &\leq \|u\|^\theta_{L^3_{t,x}([t_0-\vareps^{-\sigma},t_0]\times \R^3)} \|u\|^{1-\theta}_{L^\gamma_t([t_0-\vareps^{-\sigma},t_0], L^\rho_x)} \\
		&\lesssim \|u\|^\theta_{L^3_{t,x}([t_0-\vareps^{-\sigma},t_0]\times \R^3)} \||\nabla|^su\|^{1-\theta}_{L^\gamma_t([t_0-\vareps^{-\sigma},t_0], L^\beta_x)} \\
		&\lesssim \|u\|^\theta_{L^3_{t,x}([t_0-\vareps^{-\sigma},t_0]\times \R^3)} \|\scal{\nabla}^su\|^{1-\theta}_{L^\gamma_t([t_0-\vareps^{-\sigma},t_0], L^\beta_x)} \\
		&\lesssim \vareps^{\frac{\theta}{3}\left(\frac{1}{2}-\sigma\right)} \vareps^{-\frac{\sigma}{\gamma}(1-\theta)} \\
		&\lesssim \vareps^{\frac{\theta}{6}-\left(\frac{\theta}{3} +\frac{1-\theta}{\gamma}\right) \sigma}.
		\end{align*}
		Here we have used the fact that 
		\[
		\|\scal{\nabla} u\|_{L^\gamma_t(I, L^\beta_x)} \lesssim \scal{I}^{\frac{1}{\gamma}}
		\]
		which follows from the local theory. This proves \eqref{scat-crite-cond} by choosing $\sigma>0$ small enough. The proof is complete.
	\end{proof}

\section{Blow-up results}
\label{S4}
\setcounter{equation}{0}

In this last section, we prove the blow-up rate results as stated in Theorem \ref{prop-blow-rate}.  We start with the following upper bound for the Pohozaev functional for solutions arising from initial data in $\Ac^-_\omega$. The result of the next Lemma is also contained in \cite{BFG, FH}, but we report the short proof for sake of completeness. 

\begin{lemma}  \label{lem-A-omega-minus}
	Let $\frac{7}{3}<p<5$, $1<q<p$, and $\omega>0$. Then $\Ac^-_\omega$ is invariant under the flow of \eqref{NLS}, i.e., if $u_0 \in \Ac^-_\omega$, then $u(t) \in \Ac^-_\omega$ for all $t\in I_{\max}$. 
		In addition, we have
	\begin{align} \label{est-Ac-omega-minus}
	G(u(t)) \leq -\frac{3(p-1)}{2} (m_\omega-S_\omega(u_0)), \quad \forall t\in I_{\max}. 
	\end{align}
	Furthermore, there exists $\delta>0$ small such that 
	\begin{align} \label{est-G-delta}
	G(u(t)) + \delta \|\nabla u(t)\|^2_{L^2} \leq -\frac{3(p-1)}{4} (m_\omega-S_\omega(u_0)), \quad \forall t\in I_{\max}.
	\end{align}
\end{lemma}
\begin{proof}
	Let $u_0 \in \Ac^-_\omega$. We will show that $u(t) \in \Ac^-_\omega$ for all $t\in I_{\max}$. By the conservation of mass and energy, we have
	\begin{align} \label{S-omega-t}
	S_\omega(u(t)) = S_\omega(u_0)<m_\omega, \quad \forall t\in I_{\max}.
	\end{align}
	Assume by contradiction that there exists $t_0$ such that $G(u(t_0)) \geq 0$. As $u:I_{\max} \to H^1$ is continuous, there exists $t_1$ such that $G(u(t_1))=0$. By the definition of $m_\omega$, we have $S_\omega(u(t_1)) \geq m_\omega$ which contradicts \eqref{S-omega-t}. Thus $G(u(t)) < 0$ for all $t\in I_{\max}$ or $\Ac^-_\omega$ is invariant under the flow of \eqref{NLS}.
	
	For simplicity, we denote $u:=u(t)$ and set $f(\lambda):= S_\omega(u_\lambda)$, where $u_\lambda$ is the scaling \eqref{scaling}. We have
	\[
	f'(\lambda) = \lambda \|\nabla u\|^2_{L^2} -\frac{3(p-1)}{2(p+1)} \lambda^{\frac{3(p-1)}{2}-1} \|u\|^{p+1}_{L^{p+1}} + \frac{3(q-1)}{2(q+1)} \lambda^{\frac{3(q-1)}{2}-1} \|u\|^{q+1}_{L^{q+1}} = \frac{G(u_\lambda)}{\lambda}
	\]
	and
	\[
	\(\lambda f'(\lambda)\)' = 2\lambda \|\nabla u\|^2_{L^2} -\frac{9(p-1)^2}{4(p+1)} \lambda^{\frac{3(p-1)}{2}-1} \|u\|^{p+1}_{L^{p+1}} + \frac{9(q-1)^2}{4(q+1)} \lambda^{\frac{3(q-1)}{2}-1} \|u\|^{q+1}_{L^{q+1}}.
	\]
	We then write	
	\begin{align*}
	\(\lambda f'(\lambda)\)' = \frac{3(p-1)}{2} f'(\lambda) - \frac{3p-7}{2} \lambda \|\nabla u\|^2_{L^2} -\frac{9(q-1)(p-q)}{4(q+1)}\lambda^{\frac{3(q-1)}{2}-1} \|u\|^{q+1}_{L^{q+1}}.
	\end{align*}
	As $\frac{7}{3}<p<5$ and $1<q<p$, we get
	\begin{align} \label{est-Ac-minus}
	\(\lambda f'(\lambda)\)' \leq \frac{3(p-1)}{2} f'(\lambda), \quad \forall \lambda>0.
	\end{align}
	On the other hand, as $G(u)<0$, by Lemma \ref{lem-lamb-0}, there exists $\lambda_0 \in (0,1)$ such that $G(u_{\lambda_0}) =0$. In particular, we have $\lambda_0 f'(\lambda_0)=0$ and $f(\lambda_0) = S_\omega(u_{\lambda_0}) \geq m_\omega$. Integrating \eqref{est-Ac-minus} over $(\lambda_0,1)$, we obtain
	\[
	G(u) \leq \frac{3(p-1)}{2} (S_\omega(u)- S_\omega(u_{\lambda_0})) \leq \frac{3(p-1)}{2} (S_\omega(u)-m_\omega)
	\]
	which is \eqref{est-Ac-omega-minus}. \medskip
	
	Finally, we prove \eqref{est-G-delta}. Observe that
	\begin{align*}
	\|\nabla u\|^2_{L^2} &= \frac{6(q-1)}{3q-7} \(I_\omega(u) - \frac{p-q}{(p+1)(q-1)}\|u\|^{p+1}_{L^{p+1}} -\frac{\omega}{2}\|u\|^2_{L^2} \) \\
	&=\frac{6(q-1)}{3q-7} \(S_\omega(u) - \frac{2}{3(q-1)} G(u) -\frac{p-q}{(p+1)(q-1)}\|u\|^{p+1}_{L^{p+1}} -\frac{\omega}{2}\|u\|^2_{L^2}\).
	\end{align*}
	It follows that
	\begin{align*}
	G(u) + \delta \|\nabla u\|^2_{L^2} &= \(1-\frac{4\delta}{3q-7}\) G(u) + \frac{6\delta (q-1)}{3q-7} S_\omega(u)\\
	& - \frac{6\delta(p-q)}{(3q-7)(p+1)} \|u\|^{p+1}_{L^{p+1}} - \frac{3\delta \omega (q-1)}{3q-7} \|u\|^2_{L^2} \\
	&\leq \(1-\frac{4\delta}{3q-7}\) G(u) + \frac{6\delta (q-1)}{3q-7} S_\omega(u)
	\end{align*}
	as $\omega>0$ and $q<p$. By the energy and mass conservation laws, and \eqref{est-Ac-omega-minus}, we get
	\[
	G(u(t)) + \delta \|\nabla u(t)\|^2_{L^2} \leq - \(1-\frac{4\delta}{3q-7}\) \frac{3(p-1)}{2} (m_\omega -S_\omega(u_0)) + \frac{6\delta(q-1)}{3q-7} S_\omega(u_0),
	\]
$ \forall t\in I_{\max}$.	By taking $\delta>0$ sufficiently small, we obtain \eqref{est-G-delta}. \end{proof}
\medskip

We are now  ready to give the proof of the blow-up results in Theorem \ref{prop-blow-rate}. They are based on virial identities/estimates. We classically introduce a sufficiently smooth and decaying function  $\phi: \R^3 \to \R$. We denote the virial quantity
\[
V_\phi(t):= \int \phi |u(t)|^2 dx.
\] 
The following identities are nowadays standard (see e.g., \cite{Cazenave}):
\[
V'_\phi(t) = 2\ima \int \nabla \phi \cdot \nabla u(t) \overline{u}(t) dx
\]
and
\begin{align*}
V''_\phi(t) &= -\int \Delta^2 \phi |u(t)|^2 dx + 4 \sum_{j,k} \rea \int  \partial^2_{jk} \phi \partial_j u(t) \partial_k \overline{u}(t) dx \\
& +\frac{2(q-1)}{q+1} \int \Delta \phi |u(t)|^{q+1} dx - \frac{2(p-1)}{p+1} \int \Delta \phi |u(t)|^{p+1} dx.
\end{align*}
A first application of the above virial identities is the following virial identity for finite variance solutions. More precisely, if $u_0 \in \Sigma$, then the corresponding solution to \eqref{NLS} satisfies 
\begin{align} \label{viri-iden}
V''_{|x|^2} (t) = 8 G(u(t)), \quad \forall t\in I_{\max}.
\end{align}
Provided $u_0\in\Ac^-_\omega$, by Lemma \ref{lem-A-omega-minus}  $G(u(t))\lesssim-1$ for any $t\in I_{\max}$.  Then \eqref{viri-iden} implies finite time blow-up for solutions in $\Sigma$ via a convexity argument. \medskip

 Another applications are virial estimates for radial and cylindrical solutions \cite{BFG}. To state these estimates, we let $\theta: [0,\infty) \to [0,2]$ a smooth function satisfying
\begin{align} \label{def-theta}
\theta(r)= \left\{
\begin{array}{ccl}
2 &\text{if} &0\leq r\leq 1, \\
0 &\text{if} & r\geq 2.
\end{array}
\right.
\end{align}
We define the function $\vartheta: [0,\infty) \to [0,\infty)$ by
\[
\vartheta(r):= \int_0^r \int_0^s \theta(\tau) d\tau ds.
\]
For $\varrho>0$, we define the radial function $\phi_\varrho: \R^3 \to \R$ by
\begin{align} \label{phi-R-rad}
\phi_\varrho(x) = \phi_\varrho(r) := \varrho^2 \vartheta(r/\varrho), \quad r=|x|.
\end{align}
Then the following virial estimate for radial solutions: If $u_0 \in H^1$ is radial, then the corresponding solution to \eqref{NLS} satisfies
\begin{align} \label{viri-est-rad}
V''_{\phi_\varrho}(t) \leq 8 G(u(t)) + C\varrho^{-2} + C\varrho^{-(p-1)}  \|\nabla u(t)\|^{\frac{p-1}{2}}_{L^2}, \quad \forall t\in I_{\max}.
\end{align}
To state virial estimates for cylindrical solutions, we define, for $\varrho>0$, the function
\begin{align} \label{phi-R-cyl}
\phi_\varrho(x):= \varrho^2 \vartheta(r/\varrho) + x_3^2, \quad r:= |\overline{x}|, \quad \overline{x}=(x_1,x_2).
\end{align}
If $u_0 \in \Sigma_3$, then the corresponding solution to \eqref{NLS} satisfies
\begin{align} \label{viri-est-cyl}
V''_{\phi_\varrho}(t) \leq 8 G(u(t)) + C\varrho^{-2} + C\varrho^{-\frac{p-1}{2}}  \|\nabla u(t)\|^{p-1}_{L^2}, \quad \forall t\in I_{\max}.
\end{align}
We refer the readers to \cite{BFG} for a proof of \eqref{viri-est-rad} and \eqref{viri-est-cyl}. We are now able to prove our blow-up result for \eqref{NLS}. 

\begin{proof}[Proof of Proposition \ref{prop-blow-rate}]
	Let us start with the point \text{(i)}. The proof is based on an idea of Merle, Rapha\"el, and Szeftel \cite{MRS}. \medskip
	
\noindent 	(1) Let us consider the radial case. Using
	\[
	G(u) = \frac{3(p-1)}{2} E(u) -\frac{3p-7}{4} \|\nabla u\|^2_{L^2} - \frac{3(p-q)}{2(q+1)} \|u\|^{q+1}_{L^{q+1}},
	\] 
	we infer from \eqref{viri-est-rad} that
	\begin{align*}
	V''_{\phi_\varrho}(t) \leq 12 (p-1) E(u(t)) &-2(3p-7)\|\nabla u(t)\|^2_{L^2} - \frac{12(p-q)}{q+1} \|u(t)\|^{q+1}_{L^{q+1}} \\
	&+ C\varrho^{-2} + C\varrho^{-(p-1)} \|\nabla u(t)\|^{\frac{p-1}{2}}_{L^2}, \quad \forall t\in I_{\max},
	\end{align*}
	where $\phi_\varrho$ is as in \eqref{phi-R-rad}. As $p<5$, by Young's inequality, we have for any $\vareps>0$,
	\[
	V''_{\phi_\varrho}(t) \leq 12 (p-1) E(u(t)) -2(3p-7)\|\nabla u(t)\|^2_{L^2} + C\varrho^{-2} + \vareps \|\nabla u(t)\|^2_{L^2} + C \vareps^{-\frac{p-1}{5-p}} \varrho^{-\frac{4(p-1)}{5-p}}, 
	\]
$ \forall t\in I_{\max}$.	Taking $\vareps=3p-7$, we get
	\[
	V''_{\phi_\varrho}(t) \leq 12 (p-1) E(u(t)) -(3p-7)\|\nabla u(t)\|^2_{L^2} + C\varrho^{-2} + C \varrho^{-\frac{4(p-1)}{5-p}},
	\]
$ \forall t\in I_{\max}$.	By the conservation of energy and $2<\frac{4(p-1)}{5-p}$, we obtain
	\begin{align} \label{est-blow-rate-1}
	(3p-7) \|\nabla u(t)\|^2_{L^2} + V''_{\phi_\varrho}(t) \leq C \varrho^{-\frac{4(p-1)}{5-p}}, 
	\end{align}
	$ \forall t\in I_{\max} $ provided that $\varrho>0$ is taken sufficiently small.  Let $0<t_0<t<T^*$. We integrate \eqref{est-blow-rate-1} twice on $(t_0,t)$ and get
	\begin{align*}
	(3p-7)\int_{t_0}^t \int_{t_0}^s \|\nabla u(\tau)\|^2_{L^2} d\tau ds + V_{\phi_\varrho}(t) &\leq C \varrho^{-\frac{4(p-1)}{5-p}} (t-t_0)^2 + (t-t_0) V'_{\phi_\varrho}(t_0) + V_{\phi_\varrho}(t_0) \\
	&\leq C\varrho^{-\frac{4(p-1)}{5-p}} (t-t_0)^2 + C \varrho(t-t_0) \|\nabla u(t_0)\|_{L^2}\\
	& + C\varrho^2,
	\end{align*}
	where we have used the conservation of mass and
	\begin{align*}
	V_{\phi_\varrho}(t_0)&\leq C\varrho^2 \|u(t_0)\|^2_{L^2} \leq C \varrho^2, \\
	 V'_{\phi_\varrho}(t_0)& \leq C \varrho\|\nabla u(t_0)\|_{L^2} \|u(t_0)\|_{L^2} \leq C \varrho \|\nabla u(t_0)\|_{L^2}.
	\end{align*}
	Note that the constant $C>0$ may vary from line to line. By Fubini's Theorem, we have
	\[
	\int_{t_0}^t \int_{t_0}^s \|\nabla u(\tau)\|^2_{L^2} d\tau ds = \int_{t_0}^t \left(\int_{\tau}^t ds\right) \|\nabla u(\tau)\|^2_{L^2} d\tau = \int_{t_0}^t (t-\tau) \|\nabla u(\tau)\|^2_{L^2} d\tau.
	\]
	As $V_{\phi_\varrho}$ is non-negative, we get
	\[
	\int_{t_0}^t (t-\tau) \|\nabla u(\tau)\|^2_{L^2} d\tau \leq C\varrho^{-\frac{4(p-1)}{5-p}} (t-t_0)^2 + C \varrho(t-t_0) \|\nabla u(t_0)\|_{L^2} + C\varrho^2.
	\]
	Letting $t\nearrow T^*$, we obtain
	\[
	\int_{t_0}^{T^*} (T^*-\tau) \|\nabla u(\tau)\|^2_{L^2} d\tau \leq C\varrho^{-\frac{4(p-1)}{5-p}} (T^*-t_0)^2 + C \varrho(T^*-t_0) \|\nabla u(t_0)\|_{L^2} + C\varrho^2.
	\]
	Optimizing in $\varrho$ by choosing $\varrho^{-\frac{4(p-1)}{5-p}} (T^*-t_0)^2 = \varrho^2$ or equivalently $\varrho = (T^*-t_0)^{\frac{5-p}{p+3}}$, we have
		\begin{align*}
		\int_{t_0}^{T^*} (T^*-\tau) \|\nabla u(\tau)\|^2_{L^2} d\tau &\leq C (T^*-t_0)^{\frac{2(5-p)}{p+3}} + C(T^*-t_0)^{\frac{5-p}{p+3}+1} \|\nabla u(t_0)\|_{L^2} \\
		&\leq C(T^*-t_0)^{\frac{2(5-p)}{p+3}} + C(T^*-t_0)^2\|\nabla u(t_0)\|^2_{L^2}, 
		\end{align*}
for any $0<t_0<T^*$.		Now set 
		\begin{align} \label{5-g}
		g(t):= \int_{t}^{T^*} (T^*-\tau) \|\nabla u(\tau)\|^2_{L^2} d\tau.
		\end{align}
		We have
		\[
		g(t) \leq C(T^*-t)^{\frac{2(5-p)}{p+3}} - (T^*-t)g'(t), \quad \forall 0<t<T^*
		\]
		which is rewritten as
		\[
		\(\frac{g(t)}{T^*-t}\)' =\frac{1}{(T^*-t)^2} (g(t) + (T^*-t) g'(t)) \leq \frac{C}{(T^*-t)^{\frac{4(p-1)}{p+3}}}.
		\]
		Integrating it over $(0,t)$, we obtain
		\[
		\frac{g(t)}{T^*-t} \leq \frac{g(0)}{T^*} + \frac{C}{(T^*-t)^{\frac{3p-7}{p+3}}} -\frac{C}{(T^*)^{\frac{3p-7}{p+3}}}
		\]
		which yields
		\[
		\frac{g(t)}{T^*-t} \leq \frac{C}{(T^*-t)^{\frac{3p-7}{p+3}}} \quad \text{for } t \text{ close to } T^*.
		\]
		In particular, we have $g(t) \leq C(T^*-t)^{\frac{2(5-p)}{p+3}}$ which is \eqref{est-blow}. 
		
		To see \eqref{blow-rate}, we rewrite \eqref{est-blow} as
		\begin{align} \label{est-g}
		\frac{1}{T^*-t} \int_t^{T^*} (T^*-\tau) \|\nabla u(\tau)\|^2_{L^2} d\tau \leq \frac{C}{(T^*-t)^{\frac{3p-7}{p+3}}}.
		\end{align}
		Take $T_n \nearrow T^*$. For a fixed $n$, $g$ defined in \eqref{5-g} is a continuous function on $[T_n,T^*]$ and differentiable on $(T_n,T^*)$. By the mean value theorem, there exists $t_n \in (T_n,T^*)$ such that
		\begin{align*}
		-(T^*-t_n)\|\nabla u(t_n)\|^2_{L^2}=g'(t_n) = \frac{g(T^*)-g(T_n)}{T^*-T_n} =-\frac{\displaystyle{\int}_{T_n}^{T^*} (T^*-\tau)\|\nabla u(\tau)\|^2_{L^2} d\tau}{T^*-T_n}.
		\end{align*}
		Using \eqref{est-g}, we have
		\[
		(T^*-t_n) \|\nabla u(t_n)\|^2_{L^2} \leq \frac{C}{(T^*-T_n)^{\frac{3p-7}{p+3}}} \leq \frac{C}{(T^*-t_n)^{\frac{3p-7}{p+3}}} 
		\]
		hence
		\[
		\|\nabla u(t_n)\|^2_{L^2} \leq \frac{C}{(T^*-t_n)^{\frac{4(p-1)}{p+3}}}.
		\]
		This proves \eqref{blow-rate}. \medskip
		
	\noindent	(2) We now consider the cylindrical case. By \eqref{viri-est-cyl}, we have
		\begin{align*}
		V''_{\phi_\varrho}(t) &\leq 8G(u(t)) + C\varrho^{-2} + C\varrho^{-\frac{p-1}{2}} \|\nabla u(t)\|^{p-1}_{L^2} \\
		&\leq 12 (p-1) E(u(t)) - 2(3p-7) \|\nabla u(t)\|^2_{L^2} -\frac{12(p-q)}{q+1} \|u(t)\|^{q+1}_{L^{q+1}} \\ 
		&\quad \mathrel{\phantom{12 (p-1) E(u(t)) - 2(3p-7) \|\nabla u(t)\|^2_{L^2}}}+ C\varrho^{-2} + C\varrho^{-\frac{p-1}{2}} \|\nabla u(t)\|^{p-1}_{L^2}, \quad \forall t\in I_{\max},
		\end{align*}
		where $\phi_\varrho$ is as in \eqref{phi-R-cyl}. By Young's inequality with $p<3$ and $q<p$, we have 
		\begin{align*}
		V''_{\phi_\varrho}(t) &\leq 12(p-1) E(u(t)) - 2(3p-7) \|\nabla u(t)\|^2_{L^2} + C\varrho^{-2} + \vareps \|\nabla u(t)\|^2_{L^2} + C\vareps^{-\frac{p-1}{3-p}} \varrho^{-\frac{p-1}{3-p}} \\
		&\leq 12(p-1)E(u(t)) - (3p-7)\|\nabla u(t)\|^2_{L^2} + C\varrho^{-2} + C\varrho^{-\frac{p-1}{3-p}}, \quad \forall t\in I_{\max},
		\end{align*}
		where we have chosen $\vareps=3p-7$ to get the second line. The energy conservation and $2<\frac{p-1}{3-p}$ yield
		\[
		(3p-7)\|\nabla u(t)\|^2_{L^2} + V''_{\phi_\varrho}(t) \leq C\varrho^{-\frac{p-1}{3-p}}, \quad \forall t\in I_{\max}
		\] 
		provided that $\varrho>0$ is taken sufficiently small. By the same reasoning as above, we prove \eqref{est-blow} and \eqref{blow-rate}. \medskip

\noindent Point \text{(ii)}. 	The result is a straightforward application of the estimate in Lemma \ref{lem-A-omega-minus} yielding $G(u(t))\lesssim-1$ uniformly in time in the maximal time of existence, and the Du, Wu, and Zhang scheme \cite{DWZ}. Indeed, with respect to the NLS equation with one focusing nonlinearity, the extra defocusing term  accounts for negative contributions in the virial estimates. Hence by repeating the argument in \cite{DWZ} jointly with the uniform negative upper bound for $G$,  the proof is complete. 
\end{proof}

\begin{ackno}\rm
The authors would like to thank the anonymous referee for suggesting the reference \cite{AKN} and for valuable comments on a previous version of the paper.\\
\noindent J.B. is partially supported by project PRIN 2020XB3EFL by the Italian Ministry of Universities and Research and by  the University of Pisa, Project PRA 2022\_11. V.D.D. is supported by the European Union's Horizon 2020 Research and Innovation Programme (Grant
agreement CORFRONMAT No. 758620, PI: Nicolas Rougerie).
\end{ackno}

%\subsection*{Conflict of interest statement} The authors have no conflicts of interest to declare.
%\subsection*{Data availability statement}	
%Data sharing not applicable to this article as no datasets were generated or analysed during the current study.

\appendix

\section{Small data theory}

In this appendix, we first recall some useful tools such as dispersive and Strichartz estimates. We then prove small data global existence and small data scattering results related to \eqref{cNLS}. Let us start by reporting the well-known 3D dispersive estimate, see \cite{Cazenave} for a proof.
\begin{lemma}
	We have, for all $r\in [2,\infty]$ and for any $t\neq0$,
	\begin{align} \label{dis-est}
		\|e^{it\Delta} f\|_{L^r_x} \lesssim |t|^{-\frac{3}{2}\(1-\frac{2}{r}\)} \|f\|_{L^{r'}_x}
	\end{align}
	for any $f\in L^{r'}$.
\end{lemma}

The next ones are the well-known Strichartz estimates, arising from the dispersive estimate above. See \cite{Cazenave, KT}.
\begin{proposition} \label{prop-str-est}
	The following space-time bounds hold true.
	
	\begin{itemize}
		\item (Homogeneous Strichartz estimates)  For any $f\in L^2$ and any Schr\"odinger admissible pair $(a,b)$, i.e., 
		\[
		\frac{2}{a}+\frac{3}{b}=\frac{3}{2}, \quad b \in [2,6],
		\]
		then 
		\[
		\|e^{it\Delta} f\|_{L^a_t(\R,L^b_x)} \lesssim \|f\|_{L^2_x}.
		\]
		\item (Inhomogeneous Strichartz estimates) Let $I\subset \R$ be an interval containing 0. Then 
		\[
		\left\| \int_0^t e^{i(t-s)\Delta} F(s) ds\right\|_{L^{a}_t(I,L^b_x)} \lesssim \|F\|_{L^{\rho'}_t(I,L^{\gamma'}_x)}
		\]
		for any $F \in L^{\rho'}_t(I, L^{\gamma'}_x)$ and any Schr\"odinger admissible pairs $(a,b)$ and $(\rho,\gamma)$.
		\item (Strichartz estimates for non-admissible pairs) Let $I\subset \R$ be an interval containing 0 and $(a,b)$ be a Schr\"odinger admissible pair with $b>2$. Fix $m>\frac{a}{2}$ and define $n$ by
		\[
		\frac{1}{m} +\frac{1}{n} =\frac{2}{a}.
		\]
		Then 
		\[
		\left\| \int_0^t e^{i(t-s)\Delta} F(s) ds\right\|_{L^m_t(I, L^b_x)} \lesssim \|F\|_{L^{n'}_t(I,L^{b'}_x)}
		\]
		for any $F \in L^{n'}_t(I,L^{b'}_x)$.
	\end{itemize}
\end{proposition}

\noindent 	As in \eqref{expo-p-1} and \eqref{expo-p-2}, we introduce the following exponents:
\begin{align*}
	a_2:= \frac{4(q+1)}{3(q-1)}, \quad m_2:= \frac{2(q-1)(q+1)}{5-q},  \quad n_2:= \frac{2(q-1)(q+1)}{3q^2-5q-2},
\end{align*}
and 
\begin{align*}
	b_2:=q+1, \quad r_1:=\frac{6(q-1)(q+1)}{3q^2+2q-13}.
\end{align*}
We see that $(a_2, b_2), (m_2,r_2)$ are Schr\"odinger admissible pairs and
\[
\frac{1}{m_2} +\frac{1}{n_2} =\frac{2}{a_2},\quad \frac{1}{b_2} = \frac{1}{r_2} - \frac{\sigma_2}{3}, \quad \sigma_2:=\frac{3q-7}{2(q-1)}.
\]

The following Lemma follows directly from the above choices, H\"older's inequality, and Sobolev embeddings.
\begin{lemma} \label{lem-non-est}
	Let $I\subset \R$ be an interval. We have
	\begin{align*}
		\||u|^{p-1} u\|_{L^{n'_1}_t(I, L^{b'_1}_x)} &\lesssim \|u\|^p_{L^{m_1}_t(I, L^{b_1}_x)}, \\
		\||u|^{q-1} u\|_{L^{n'_2}_t(I, L^{b'_2}_x)} &\lesssim \|u\|^q_{L^{m_2}_t(I, L^{b_2}_x)}, \\
		\|\scal{\nabla}(|u|^{p-1} u)\|_{L^{a'_1}_t(I, L^{b'_1}_x)} &\lesssim \|u\|^{p-1}_{L^{m_1}_t(I, L^{b_1}_x)} \|\scal{\nabla} u\|_{L^{a_1}_t(I, L^{b_1}_x)}, \\
		\|\scal{\nabla}(|u|^{q-1} u)\|_{L^{a'_2}_t(I, L^{b'_2}_x)} &\lesssim \|u\|^{q-1}_{L^{m_2}_t(I, L^{b_2}_x)} \|\scal{\nabla} u\|_{L^{a_2}_t(I, L^{b_2}_x)}, \\
		\|u\|_{L^{m_1}_t(I, L^{b_1}_x)} &\lesssim \||\nabla|^{\sigma_1} u\|_{L^{m_1}_t(I, L^{r_1}_x)} \lesssim \|\scal{\nabla} u\|_{L^{m_1}_t(I, L^{r_1}_x)}, \\
		\|u\|_{L^{m_2}_t(I, L^{b_2}_x)} &\lesssim \||\nabla|^{\sigma_2} u\|_{L^{m_2}_t(I, L^{r_2}_x)}\lesssim \|\scal{\nabla} u\|_{L^{m_2}_t(I, L^{r_2}_x)}.
	\end{align*}
\end{lemma}
We next prove a global existence result for small data. 	
\begin{lemma} \label{lem-small-gwp}
	Let $\frac{7}{3}<q<p<5$ and $T>0$ be such that $u(T)\in H^1$. Then there exists $\delta>0$ sufficiently small such that if
	\[
	\|e^{i(t-T)\Delta} u(T)\|_{L^{m_1}_t([T,\infty), L^{b_1}_x) \cap L^{m_2}_t([T,\infty), L^{b_2}_x)} <\delta,
	\]
	then there exists a unique solution to \eqref{NLS} with initial datum $u(T)$ satisfying
	\[
	\|u\|_{L^{m_1}_t([T,\infty), L^{b_1}_x) \cap L^{m_2}_t([T,\infty), L^{b_2}_x)} \leq 2\|e^{i(t-T)\Delta} u(T)\|_{L^{m_1}_t([T,\infty), L^{b_1}_x) \cap L^{m_2}_t([T,\infty), L^{b_2}_x)} 
	\]
	and
	\[
	\|\scal{\nabla} u\|_{L^{a_1}_t([T,\infty), L^{b_1}_x) \cap L^{a_2}([T,\infty), L^{b_2}_x)} \leq C\|u(T)\|_{H^1_x}
	\]
	for some constant $C>0$. 
\end{lemma}

\begin{proof}
	We consider
	\[
	X_T:= \left\{ u : \|u\|_{L^{m_1}_t(I,L^{b_1}_x) \cap L^{m_2}_t(I,L^{b_2}_x)} \leq M, \ \|\scal{\nabla} u\|_{L^{a_1}_t(I,L^{b_1}_x) \cap L^{a_2}_t(I,L^{b_2}_x)} \leq L\right\}
	\]
	equipped with the distance
	\[
	d(u,v) := \|u-v\|_{L^{a_1}_t(I,L^{b_1}_x) \cap L^{a_2}_t(I,L^{b_2}_x)},
	\]
	where $I=[T,\infty)$ and $M,L>0$ will be chosen later. By the persistence of regularity (see e.g., \cite[Theorem 1.2.5]{Cazenave}), we readily see that $(X_T,d)$ is a complete metric space. Our purpose is to show that the Duhamel functional
	\begin{align} \label{duha-form}
		\Phi_T(u(t)) := e^{i(t-T)\Delta} u(T) + i\int_T^t e^{i(t-s)\Delta} \(|u(s)|^{p-1}u(s) - |u(s)|^{q-1} u(s)\) ds
	\end{align}
	is a contraction on $(X_T, d)$. By Strichartz estimates and Lemma \ref{lem-non-est}, we have
	\begin{align*}
		\|\Phi_T(u)\|_{L^{m_1}_t(I,L^{b_1}_x)} &\leq \|e^{i(t-T)\Delta} u(T)\|_{L^{m_1}_t(I,L^{b_1}_x)} + \left\| \int_T^t e^{i(t-s)\Delta} |u(s)|^{p-1} u(s) ds \right\|_{L^{m_1}_t(I,L^{b_1}_x)} \\
		& + \left\| \int_T^t e^{i(t-s)\Delta} |u(s)|^{q-1} u(s) ds \right\|_{L^{m_1}_t(I,L^{b_1}_x)} \\
		&\leq  \|e^{i(t-T)\Delta} u(T)\|_{L^{m_1}_t(I,L^{b_1}_x)} + C\||u|^{p-1} u\|_{L^{n'_1}_t(I,L^{b'_1}_x)} \\
		& + \left\| \scal{\nabla} \int_T^t e^{i(t-s)\Delta} |u(s)|^{q-1} u(s)ds\right\|_{L^{m_1}_t(I,L^{r_1}_x)} \\
		&\leq  \|e^{i(t-T)\Delta} u(T)\|_{L^{m_1}_t(I,L^{b_1}_x)} + C\||u|^{p-1} u\|_{L^{n'_1}_t(I,L^{b'_1}_x)} \\
		&+ C\|\scal{\nabla}(|u|^{q-1} u)\|_{L^{a'_2}_t(I,L^{b'_2}_x)} \\
		&\leq  \|e^{i(t-T)\Delta} u(T)\|_{L^{m_1}_t(I,L^{b_1}_x)} + C\|u\|^p_{L^{m_1}_t(I,L^{b_1}_x)} \\
		&+ \|u\|^{q-1}_{L^{m_2}_t(I,L^{b_2}_x)} \|\scal{\nabla} u\|_{L^{a_2}_t(I,L^{b_2}_x)}. 
	\end{align*}
	Similarly, we have
	\[
	\begin{aligned}
		\|\Phi_T(u)\|_{L^{m_2}_t(I,L^{b_2}_x)} &\leq \|e^{i(t-T)\Delta} u(T)\|_{L^{m_2}_t(I,L^{b_2}_x)}+ C\|u\|^{p-1}_{L^{m_1}_t(I,L^{b_1}_x)} \|\scal{\nabla}u\|_{L^{a_2}_t(I,L^{b_2}_x)}\\
		&  + C\|u\|^q_{L^{m_2}_t(I,L^{b_2}_x)}.
	\end{aligned}
	\]
	We next estimate
	\begin{align*}
		\|\scal{\nabla}\Phi_T(u)\|_{L^{a_1}_t(I,L^{b_1}_x)} &\leq C\|u(T)\|_{H^1_x} + C\|\scal{\nabla}(|u|^{p-1}u)\|_{L^{a'_1}_t(I,L^{b'_1}_x)} \\
		&+ C\|\scal{\nabla}(|u|^{q-1} u)\|_{L^{a'_2}_t(I,L^{b'_2}_x)} \\
		&\leq C\|u(T)\|_{H^1_x} + C\|u\|^{p-1}_{L^{m_1}_t(I,L^{b_1}_x)} \|\scal{\nabla} u\|_{L^{a_1}_t(I,L^{b_1}_x)} \\
		&+ C \|u\|^{q-1}_{L^{m_2}_t(I,L^{b_2}_x)} \|\scal{\nabla} u\|_{L^{a_2}_t(I,L^{b_2}_x)}
	\end{align*}
	and
	\[
	\begin{aligned}
		\|\scal{\nabla}\Phi_T(u)\|_{L^{a_2}_t(I,L^{b_2}_x)}& \leq  C\|u(T)\|_{H^1_x} + C\|u\|^{p-1}_{L^{m_1}_t(I,L^{b_1}_x)} \|\scal{\nabla} u\|_{L^{a_1}_t(I,L^{b_1}_x)}\\
		& + C \|u\|^{q-1}_{L^{m_2}_t(I,L^{b_2}_x)} \|\scal{\nabla} u\|_{L^{a_2}_t(I,L^{b_2}_x)}.
	\end{aligned}
	\]
	Next we have
	\begin{align*}
		\|\Phi_T(u)-\Phi_T(v)\|_{L^{a_1}_t(I,L^{b_1}_x) \cap L^{a_2}_t(I,L^{b_2}_x)} &\leq C\||u|^{p-1}u-|v|^{p-1} v\|_{L^{n'_1}_t(I,L^{b'_1}_x)}\\
		& + C\||u|^{q-1}u-|v|^{q-1}v\|_{L^{n'_2}_t(I,L^{b'_2}_x)} \\
		&\leq \Big(\|u\|^{p-1}_{L^{m_1}_t(I,L^{b_1}_x)} + \|v\|^{p-1}_{L^{m_1}_t(I,L^{b_1}_x)} \Big)\|u-v\|_{L^{a_1}_t(I,L^{b_1}_x)} \\
		&+ \Big(\|u\|^{q-1}_{L^{m_2}_t(I,L^{b_2}_x)} + \|v\|^{q-1}_{L^{m_2}_t(I,L^{b_2}_x)} \Big)\|u-v\|_{L^{a_2}_t(I,L^{b_2}_x)}.
	\end{align*}
	Thus there exists $C>0$ independent of $T$ such that for any $u,v\in X_T$, we have
	\begin{align*}
		\|\Phi_T(u)\|_{L^{m_1}_t(I,L^{b_1}_x) \cap L^{m_2}_t(I,L^{b_2}_x)} &\leq \|e^{i(t-T)\Delta}u(T)\|_{L^{m_1}_t(I,L^{b_1}_x) \cap L^{m_2}_t(I,L^{b_2}_x)} + C (M^q+M^p) \\
		&+ C (M^{q-1}+M^{p-1}) L, \\
		\|\scal{\nabla}\Phi_T(u)\|_{L^{a_1}_t(I,L^{b_1}_x) \cap L^{a_2}_t(I,L^{b_2}_x)} &\leq C\|u(T)\|_{H^1_x} + C(M^{q-1}+M^{p-1}) L, 
	\end{align*}
	and
	\[
	d(\Phi_T(u),\Phi_T(v))\leq C(M^{q-1}+M^{p-1}) d(u,v).
	\]
	
	\noindent By choosing $M=2\|e^{i(t-T)\Delta} u(T)\|_{L^{m_1}_t(I,L^{b_1}_x)\cap L^{m_2}_t(I,L^{b_2}_x)}$ and $L=2C\|u(T)\|_{H^1_x}$ and taking $M$ sufficiently small, we see that $\Phi_T$ is a contraction on $(X_T,d)$. This completes the proof.
\end{proof}
The next Lemma is a small data scattering result. 	
\begin{lemma} \label{lem-small-scat}
	Let $\frac{7}{3}<q<p<5$. Suppose that $u(t)$ is a global solution to \eqref{NLS} satisfying $\|u\|_{L^\infty_t(\R, H^1_x)} <\infty$. Then there exists $\delta>0$ sufficiently small such that if 
	\[
	\|e^{i(t-T)\Delta} u(T)\|_{L^{m_1}_t([T,\infty), L^{b_1}_x)} <\delta
	\]
	for some $T>0$, then $u$ scatters forward in time.
\end{lemma}
\begin{proof}
	We first observe that for any interval $I\subset \R$,
	\begin{align*}
		\|u\|_{L^{m_2}_t(I, L^{b_2}_x)} \leq \|u\|^\theta_{L^{m_1}_t(I,L^{b_1}_x)} \|u\|^{1-\theta}_{L^{\rho}_t(I,L^{\gamma}_x)},
	\end{align*}
	where
	\[
	\theta=\frac{(3q-7)(p-1)}{(3p-7)(q-1)}, \quad \rho=\frac{(1-\theta) m_1 m_2}{m_1-\theta m_2}, \quad \gamma=\frac{(1-\theta) b_1 b_2}{b_1-\theta b_2}. 
	\]
	We readily check that for $\frac{7}{3}<q<p<5$,
	\[
	\theta \in (0,1), \quad \frac{2}{\rho}+\frac{3}{\gamma}=\frac{3}{2}, \quad \gamma \in [2,6],
	\]
	namely $(\rho,\gamma)$ is a Schr\"odinger admissible pair. It follows that
	\begin{align*}
		\|e^{i(t-T)\Delta} u(T)\|_{L^{m_2}_t([T,\infty), L^{b_2}_x)} &\leq \|e^{i(t-T)\Delta} u(T)\|_{L^{m_1}_t([T,\infty), L^{b_1}_x)}^{\theta} \|e^{i(t-T)\Delta} u(T)\|^{1-\theta}_{L^{\rho}_t([T,\infty), L^{\gamma}_x)} \\
		&\leq \delta^\theta \|u(T)\|_{H^1_x}^{1-\theta}.
	\end{align*}
	Thus
	\[
	\|e^{i(t-T)\Delta} u(T)\|_{L^{m_1}_t([T,\infty), L^{b_1}_x) \cap L^{m_2}_t([T,\infty), L^{b_2}_x)} <\vareps(\delta)
	\]
	for some $\vareps(\delta)>0$ small depending on $\delta$. By Lemma \ref{lem-small-gwp}, we have
	\begin{align} \label{est-small-1}
		\|u\|_{L^{m_1}_t([T,\infty), L^{b_1}_x) \cap L^{m_2}_t([T,\infty), L^{b_2}_x)} \leq 2\|e^{i(t-T)\Delta} u(T)\|_{L^{m_1}_t([T,\infty), L^{b_1}_x) \cap L^{m_2}_t([T,\infty), L^{b_2}_x)} 
	\end{align}
	and
	\begin{align} \label{est-small-2}
		\|\scal{\nabla} u\|_{L^{a_1}_t([T,\infty), L^{b_1}_x) \cap L^{a_2}([T,\infty), L^{b_2}_x)} \leq C\|u(T)\|_{H^1_x}
	\end{align}
	for some constant $C>0$. \\
\medskip
\noindent	For $T<t_1<t_2$, we use the Duhamel formula \eqref{duha-form}, \eqref{est-small-1}, and \eqref{est-small-2} to have
	\begin{align*}
		\|e^{-it_2\Delta} u(t_2)-e^{-it_1\Delta} u(t_1)\|_{H^1_x} &\leq C\|\scal{\nabla} (|u|^{p-1}u)\|_{L^{a'_1}_t((t_1,t_2), L^{b'_1}_x)}\\
		& + C\|\scal{\nabla} (|u|^{q-1}u)\|_{L^{a'_2}_t((t_1,t_2), L^{b'_2}_x)} \\
		&\leq C\|u\|^{p-1}_{L^{m_1}_t((t_1,t_2), L^{b_1}_x)} \|\scal{\nabla} u\|_{L^{a_1}_t((t_1,t_2), L^{b_1}_x)} \\
		& + C \|u\|^{q-1}_{L^{m_2}_t((t_1,t_2), L^{b_2}_x)} \|\scal{\nabla} u\|_{L^{a_2}_t((t_1,t_2), L^{b_2}_x)} \to 0
	\end{align*}
	as  $t_1, t_2 \to \infty$.	This shows that $\{e^{-it\Delta} u(t)\}_{t\to \infty}$ is a Cauchy sequence in $H^1_x$. Thus there exists 
	\[
	u_+ = e^{-iT\Delta} u(T) + i \int_T^\infty e^{-is\Delta} \(|u(s)|^{p-1}u(s)-|u(s)|^{q-1}u(s)\) ds \in H^1_x
	\] 
	such that $e^{-it\Delta} u(t) \to u_+$ strongly in $H^1_x$ as $t\to \infty$. By the unitary property of the propagator, we obtain		
	\[
	\|u(t)-e^{it\Delta} u_+\|_{H^1_x} \to 0 \text{ as } t\to \infty.
	\]
	The proof is complete.
\end{proof}

%	\section*{Bibliography}

\end{document}